\documentclass[final,leqno]{siamltex}
\pdfoutput=0
\usepackage{epsfig,multirow,bbm}
\usepackage{cancel}
\usepackage{subfigure}
\usepackage{gensymb}
\usepackage{amsmath}
\usepackage{amssymb}
\usepackage{dsfont}
\usepackage{url,changebar}
\usepackage{mdwlist}
\usepackage{cite}   
\usepackage[pdf]{pstricks}
\usepackage[bookmarks=false,colorlinks=true,linkcolor={blue},pdfstartview={XYZ null null 1.22}]{hyperref}

\usepackage{boxedminipage}

\graphicspath{{./}}

\newcommand{\dt}{{\Delta t}}
\newcommand{\dT}{{\Delta T}}
\newcommand{\tx}{t}
\newcommand{\Tx}{T}

\newcommand{\GLMM}{\mathbb{M}}
\newcommand{\GLMF}{\mathbb{F}}
\newcommand{\GLMS}{\mathbb{S}}

\def\GLyy{{\rm GLy\tilde{y}}}
\def\GLye{{\rm GLy}\varepsilon}
\def\GLMMyy{{\rm \mathbb{M}_{y\tilde{y}}}}
\def\GLMMye{{\rm \mathbb{M}_{y\varepsilon}}}

\newcommand\ST{\rule[-0.75em]{0pt}{2em}}

        {\begin{list}{\labelitemi}{
                \setlength{\topsep}{0pt}
                \setlength{\parskip}{0pt}
                \setlength{\itemsep}{0pt}
                \setlength{\parsep}{0pt}
                \setlength{\leftmargin}{23pt}
                \setlength{\labelwidth}{23pt}
        }}
        {\end{list}}

        {\begin{list}{\labelenumi}{
                \usecounter{enumi}
                \setlength{\topsep}{0pt}
                \setlength{\parskip}{0pt}
                \setlength{\itemsep}{0pt}
                \setlength{\parsep}{0pt}
                \setlength{\leftmargin}{\parindent}
                \setlength{\labelwidth}{\parindent}
        }}
        {\end{list}}

\newcommand{\ones}{\mathbbm{1}}

\def\RKA{{\mathcal A}}
\def\RKb{{\mathcal B}}
\def\RKc{{\mathcal C}}

\def\DRKB{B^*} 
\def\DRKD{D^*} 
\def\ODRKB{\overline{B}^*} 
\def\ODRKD{\overline{D}^*} 

\def\GLMA{\mathbf{A}}
\def\GLMU{\mathbf{U}}
\def\GLMB{\mathbf{B}}
\def\GLMV{\mathbf{V}}

\newcommand{\TreeIndexSet}{{\mathcal T}}
\newcommand{\TreeSet}{{T}}
\newcommand{\Ttree}{{\tau}}
\newcommand{\LabeledTreeSet}{{LT}}
\newcommand{\TreeOrder}{{\rho}}

\def\BTreeIA{
\begin{pspicture}(0,0)(12pt,12pt)
\psdot(5pt,6pt)
\end{pspicture} 
}

\def\BTreeIIA{
\begin{pspicture}(0,0)(12pt,12pt)
\psdot(5pt,0pt)
\psdot(5pt,6pt)
\psline(5pt,0)(5pt,6pt)
\end{pspicture} 
}

\def\BTreeIIIA{
\begin{pspicture}(0,0)(12pt,12pt)
\psdot(5pt,0pt)
\psdot(0pt,6pt)
\psdot(10pt,6pt)
\psline(5pt,0)(0pt,6pt)
\psline(5pt,0)(10pt,6pt)
\end{pspicture} 
}

\def\BTreeIVA{
\begin{pspicture}(0,0)(12pt,14pt)
\psdot(5pt,0pt)
\psdot(5pt,6pt)
\psdot(5pt,12pt)
\psline(5pt,0)(5pt,6pt)
\psline(5pt,6pt)(5pt,12pt)
\end{pspicture} 
}

\def\BTreeVA{
\begin{pspicture}(0,0)(12pt,18pt)
\psdot(5pt,0pt)
\psdot(5pt,6pt)
\psdot(0pt,6pt)
\psdot(10pt,6pt)
\psline(5pt,0pt)(5pt,6pt)
\psline(5pt,0pt)(0pt,6pt)
\psline(5pt,0pt)(10pt,6pt)
\end{pspicture} 
}

\def\BTreeVIA{
\begin{pspicture}(0,0)(12pt,18pt)
\psdot(5pt,0pt)
\psdot(5pt,12pt)
\psdot(0pt,6pt)
\psdot(10pt,6pt)
\psline(5pt,0)(10pt,6pt)
\psline(10pt,6pt)(5pt,12pt)
\psline(5pt,0pt)(0pt,6pt)
\end{pspicture} 
}

\def\BTreeVIIA{
\begin{pspicture}(0,0)(12pt,18pt)
\psdot(5pt,0pt)
\psdot(5pt,6pt)
\psdot(0pt,12pt)
\psdot(10pt,12pt)
\psline(5pt,0)(5pt,6pt)
\psline(5pt,6pt)(0pt,12pt)
\psline(5pt,6pt)(10pt,12pt)
\end{pspicture} 
}

\def\BTreeVIIIA{
\begin{pspicture}(0,0)(12pt,18pt)
\psdot(5pt,0pt)
\psdot(5pt,5pt)
\psdot(5pt,10pt)
\psdot(5pt,15pt)
\psline(5pt,0)(5pt,5pt)
\psline(5pt,5pt)(5pt,10pt)
\psline(5pt,10pt)(5pt,15pt)
\end{pspicture} 
}

\title{Estimating Global Errors in Time Stepping\thanks{This
    material is based upon work supported by the U.S.\ Department of
    Energy, Office of Science, Advanced Scientific
    Computing Research Program under contract DE-AC02-06CH11357, FWP
    \#56706 and \#57K87.}}   

\author{Emil Constantinescu\thanks{Mathematics and Computer Science
    Division, Argonne National Laboratory, 9700 S. Cass Avenue,
    Argonne, IL 60439 ({\tt emconsta@mcs.anl.gov}).}} 

\begin{document}

\maketitle

\begin{abstract}
This study introduces new time-stepping strategies with built-in
global error estimators. The new methods propagate the defect along
with the numerical solution much like solving for the correction
or Zadunaisky's procedure;
however, the proposed approach allows for overlapped internal
computations and, therefore, represents a generalization of the
classical numerical schemes for solving differential equations with
global error estimation. The resulting algorithms can be effectively
represented as general linear methods. We present a few explicit
self-starting schemes akin to Runge-Kutta methods with global error
estimation and illustrate the theoretical considerations on several
examples. 
\end{abstract}



\pagestyle{myheadings}
\thispagestyle{plain}
\markboth{E.M. Constantinescu}{Estimating Global Errors in Time Stepping}


%
\section{Introduction}
%

The global error or {\em a posteriori} error represents the actual numerical
error resulting after applying a time-stepping algorithm. Calculating
this error and controlling it by step-size reduction are
generally viewed as expensive processes, and therefore in practice
only local error or the error from one step to the next is used for
error estimation and control
\cite{Sundials,petsc-user-ref-e,Trilinos,MATLAB}. In general, however,
local error estimates cannot predict how those local errors will
propagate through the simulation, and for some problems these local
errors can grow to be larger than intended. Therefore, 
from the end-user perspective, local error estimation (LEE) is not
always suitable, 
especially for problems with unstable modes or long integration times
\cite{Higham_1991,Hairer_B2008_I} because the solution may end up
having unexpectedly large numerical errors. This aspect prompts 
us to revisit global error estimation in order to make it more
transparent and practical, which ultimately leads to better error
control and reliable accuracy.  

In this study we introduce and analyze efficient
strategies for estimating global errors for 
time-stepping algorithms.  We present a unifying approach that
includes most of the classical strategies as particular cases, and we 
develop new algorithms that fall under general linear time-stepping
schemes. One of the most comprehensive surveys for global error
estimation is by Skeel \cite{Skeel_1986}. We focus on a subset of the
methods discussed therein and generalize some of the results presented
there. 

Global error estimation in time stepping has a long history
\cite{Jeannerod_1998,Aid_1997,Scholz_1989,Richert_1987,Merluzzi_1978,Shampine_1986b,Dormand_1989,Dormand_1984,Dormand_1985,Dormand_1994,Makazaga_2003,Prince_1978,Rive_1975,Shampine_2005b,Lang_2007,Enright_1989,Utumi_1996,Estep_1995}. {\em
  A
posteriori} global error estimation has been recently discussed in
\cite{Cao_2005,Estep_2005,Kulikov_2012,Banks_2012}. Step-size control
with multimethod Runge-Kutta (RK) is discussed in
\cite{Shampine_TR1979,Shampine_1986a,Shampine_1986b,Cash_1990}. Global
error estimation for stiff problems is discussed in
\cite{Dahlquist_1982,Kulikov_2013,Stetter_1974,Stetter_1982}. Adjoint
methods 
for global error estimation for PDEs \cite{Berzins_1988,Lawson_1991} are analyzed in
\cite{Giles_2002,Connors_2013}. These studies cover most of the
types of strategy that have been proposed to address global error
estimation. The Zadunaisky procedure  \cite{Zadunaisky_1976} and
the related procedure for solving for the correction \cite{Skeel_1986} are 
arguably the most popular global-error estimation strategies \cite{Aid_1997,Jeannerod_1998}. The work
of Dormand et al. \cite{Dormand_1984,Dormand_1985} relies on this
procedure and is extended to a composition of RK methods in
\cite{Dormand_1994}. Further extensions are introduced by Murua and Makazaga
\cite{Murua_2000,Makazaga_2003}. Shampine \cite{Shampine_1986b} proposes
using multiple methods to estimate global errors.  

Recent work on global error control by Kulikov, Weiner, et al. 
\cite{Weiner_2012a,Kulikov_2010b,Kulikov_2010} extends the quasi-consistency property
introduced by Skeel \cite{Skeel_1976} and recently advanced by Kulikov
\cite{Kulikov_2009}. Moreover, these
ideas were extended to peer methods by using a peer-triplets strategy by
Weiner and Kulikov  
\cite{Weiner_2014}. These
strategies seem to give very good results in terms of global
error control on prototypical problems. The general linear (GL)-based algorithm proposed in
this study bears a more general representation of the methods
discussed above. We will also show how Runge-Kutta triplets
\cite{Dormand_1985} and by proxy the peer triplets \cite{Weiner_2014}
are naturally represented as GL methods.

Our work builds on similar ideas introduced by Shampine
\cite{Shampine_1986b} and  
Zadunaisky \cite{Zadunaisky_1976} and the followups in the sense that
the strategy evolves the defect along with 
the solution; however, in our strategy the internal calculations of
the two quantities can be overlapped by using a single scheme to
evolve them simultaneously. Therefore, the new method automatically integrates
the local truncation error or defect. Previous strategies can be cast as
particular cases of the one introduced in this study when the overlapping
part is omitted. This leads to new types of
schemes that are naturally represented as GL methods,
which are perfectly suited for this strategy, as we demonstrate.

The general linear methods introduced in this study have built-in
asymptotically correct global and local error estimators. These
methods propagate at least two quantities; one of them is the solution,
and the other one can be either the global error or equivalently
another solution that can be used to determine the global errors. We
show that two new elements are required for GL methods to propagate
global errors: $(i)$ a particular 
relation between the truncation error of the two quantities and $(ii)$
a decoupling property between the errors of its two outputs. The
GL framework encapsulates all linear time-stepping algorithms and
provides a platform for robust algorithms with built-in error
estimates. Moreover, this encapsulated treatment simplifies the analysis
of compound schemes used for global error analysis; for instance,
stability analysis turns out to be much simpler in this representation.

We consider the first-order system of nonautonomous
ordinary differential equations 
\begin{align}
\label{eq:ODE}
&y(\tx)'=f(\tx,y(\tx))\,;~ y(\tx_0)=y_0\,,~ \tx_0 < \tx \le \Tx\,, ~ y\in
\mathbb{R}^{m}, f:\mathbb{R}^{m+1} \rightarrow \mathbb{R}^{m}\,,
\end{align}
of size $m$ with $y_0$ given. We will use the tensor notation denoting
the components in \eqref{eq:ODE} by $y^{\{j\}}$ and $f^{\{j\}}$,
$j=1,\,2,\,  \dots ,\, m$. We will often consider nonautonomous systems
because the exposition is less cluttered. In order to convert \eqref{eq:ODE} to
autonomous form, the system can be augmented with
$(y^{\{m+1\}})'=1$, with $y^{\{m+1\}}(\tx_0)=\tx_0$; hence,
$\tx=y^{\{m+1\}}(\tx)$.
This is likely not a restrictive theoretical assumption,
but there can be exceptions \cite{Oliver_1975}; however, in practice
it is preferable to treat the temporal components
separately. For brevity, we will
refer to \eqref{eq:ODE} in both autonomous and nonautonomous forms
depending on the context. 

The purpose of this study is to analyze strategies for estimating the
global error at every time step $n$,
\begin{align}
&\varepsilon(\tx_n)=y(\tx_n)-y_n \,, ~ n=1,2,
\dots, N_T\,,
\label{eq:ODE:GlobalError}
\end{align}
that is, the difference between the exact solution $y(\tx_n)$ and a numerical
approximation $y_n$ from a sequence of $N_T$ steps. {\em A priori} and
{\em a posteriori} error bounds under appropriate 
smoothness assumptions are well known
\cite{Henrici_1962,Hairer_B2008_I}. This study focuses on efficient
     {\em a posteriori} estimates of $\varepsilon(\tx_n)$.  

This study has several limitations. Stiff
differential equations are not directly addressed. Although the
theory presented here, more precisely the consistency result, applies
to the stiff case; however, additional
constraints are believed to be necessary to preserve the asymptotic
correctness of the error estimators. Moreover, we do not discuss
global error control.  Although algorithms introduced in this study
work well with variable time steps as we illustrate through the
results in Fig. \ref{fig:nr:err:adapt}, we do not address error control
strategies here. This is a topic for a future study; however,
local error control can be used as before with global error estimates
being a diagnostic of the output.

We aim to bring a self-contained view of global error estimation. New
results are interlaced with classical theory to provide a contained
picture for this topic. We also illustrate the connection among
different strategies. The proposed algorithm generalizes all the
strategies reviewed in this study and provides a robust instrument for
estimating {\em a posteriori} errors in numerical integration. Section
\ref{sec:Global:Errors} introduces the background for the 
theoretical developments and 
discusses different
strategies to estimate the global errors, which include developments
that form the basis of the proposed approach. 
In Sec. \ref{sec:GLMs} we discuss the
general linear methods that are used to represent practical
algorithms. The analysis of these schemes and examples are provided in
Sec. \ref{sec:Methods:GEM}. In Sec. \ref{sec:relations} we discuss the relationship between the
approach introduced here and related strategies and show how the
latter are particular instantiations of the former. Several numerical experiments are
presented in Sec. \ref{sec:Numerical:Exeriments}, and concluding
remarks are discussed in Sec. \ref{sec:discussion}. 

%
\section{Global error estimation\label{sec:Global:Errors}}
%

Let us consider a one-step linear numerical discretization method for
\eqref{eq:ODE}, 
\begin{align}
\label{eq:Henrici:OneStepNotation}
y_{n+1} = y_{n} + \dt \Phi(\tx_{n},y_{n},\dt_n)\,, ~y_0=y(\tx_0)\,,~ n=1,2,
\dots , N_T \,,
\end{align} 
where $\Phi$ is called the Taylor increment function with
$\Phi(\tx_{n},y_{n},0)=f(\tx_n,y_n)$. We denote the time series
obtained via \eqref{eq:Henrici:OneStepNotation} with step $\dt$ by $\{y_\dt\}$.
A method of order $p$ for a sufficiently smooth function $f$ satisfies 
\begin{subequations}
\label{eq:local:error:bound}
\begin{align}
\label{eq:local:error:bound:1}
||y(\tx_n+\dt) - y_{n+1} || \le  C_1 \dt^{p+1}\,, y_n = y(\tx_n)\,,
\end{align}
for a constant $C_1$. The local error then satisfies 
\begin{align}
\label{eq:local:error:expansion}
y(\tx + \dt) - y(\tx) - \dt \Phi(\tx,y(\tx),\dt)  =
  d_{p+1}(\tx)
\dt^{p+1} + 
{\mathcal O}(\dt^{p+2}) \,.
\end{align}
\end{subequations} 
The following classical result states the bounds on the global errors.
\begin{theorem} 
\label{th:glb:err:bnds} 
Let $U$ be a neighborhood of $\{(\tx,y(\tx)) | \tx_0 \le \tx \le T \}$,
where  $y(\tx)$ is the exact solution of \eqref{eq:ODE} and there
exists a constant $L$ such that $|| f(\tx,y)-f(\tx,z) || \le L  ||y-z|| $ and
\eqref{eq:local:error:bound} is satisfied for $(\tx,x)$, $(\tx,y) \in
U$. Then 
\begin{align}
\label{eq:global:error:bounds}
|| \varepsilon(t) || \le \dt^p \frac{C_2}{L} \left(e^{L(\tx -
  \tx_0)} - 1 \right) \, 
\end{align}
for a constant $C_2$.
\end{theorem} 

This is proved in several treatises
\cite{Henrici_1962,Spijker_1971,Hairer_B2008_I}. Under
sufficient smoothness assumptions \cite{Henrici_1962,Hairer_1984}, it
follows that the 
{\em global error} satisfies   
\begin{align}
\label{eq:global:error:general}
\varepsilon(\tx) = y(\tx) - y_{\dt}(\tx) = e_{p}(\tx) \dt^p + {o}(\dt^p)\,,
\end{align}
where $y_{n} := y_{\dt}(\tx)$ at $\tx = \tx_{0} + n \dt$. 
These results are obtained by comparing the expansions of the exact
and the numerical solutions. To alleviate the analysis difficulties that come with large $p$, we use the B-series representation of the derivatives. 
\begin{definition}[\rm Rooted trees and labeled trees
    \cite{Butcher_1963,Chartier_2010}\label{def:LT}] 
Let $\TreeIndexSet$ be a set of ordered indexes $\TreeIndexSet_q=\{j_1 <
j_2 < j_3 < \cdots < j_q\}$ with cardinality $q$. A labeled tree of
order $q$ is a
mapping $\Ttree  : \TreeIndexSet_q \backslash \{j_1\} \rightarrow
\TreeIndexSet_q$ such that $\Ttree(j)<j$, $\forall j \in
\TreeIndexSet_q\backslash \{j_1\}$. The set of all labeled trees of
order $q$ is denoted by $\LabeledTreeSet_q$. The order of a tree is
denoted by $\TreeOrder(\Ttree)=q$. Furthermore, we define an
equivalence class of order $q$ as the permutation
$\sigma:\TreeIndexSet_q  \rightarrow \TreeIndexSet_q$ such that
$\sigma(j) = j$, $\Ttree_k \sigma = \sigma \Ttree_\ell$,
$\Ttree_k,\Ttree_\ell \in \LabeledTreeSet_q$. These unlabeled trees of
order $q$ are denoted by $\TreeSet_q$, and the number of different
monotonic labelings 
of $\Ttree \in  \TreeSet_q$ is denoted by $\alpha(\Ttree)$.  Also,
$\TreeSet_q^\#=\TreeSet_q \cup {\emptyset}$, where $\emptyset$ is the empty tree
and the only one with  $\TreeOrder(\emptyset)=0$.
\end{definition}

\begin{definition}[Elementary differentials
    \cite{Butcher_1963,Chartier_2010}\label{def:ElDiff}] 
For a labeled tree  $\Ttree \in \LabeledTreeSet_q $ we call an
elementary differential the expression
\begin{align}
F^{\{K_1\}} (\Ttree)(y) = \sum_{K_2,K_3,\dots,K_{q}} \prod_{i=1}^q
f^{\{K_i\}}_{\Ttree^{-1}(K_i)}\,,
\end{align}
where $K_1,K_2,\dots,K_q = 1,2,\dots,m$, and 
$f^{\{J\}}_{K_1,K_2,\dots,K_r}=\partial^r f^{\{J\}}/ \partial
y^{\{K_1\}}y^{\{K_2\}} \dots y^{\{K_r\}} $. We denote by
$F(\Ttree)(y)=[F^{\{1\}} (\Ttree)(y) ,F^{\{2\}} (\Ttree)(y) , \dots,
  F^{\{m\}} (\Ttree)(y) ]^T$. 
\end{definition}

\noindent{}We use the graphical notation to represent derivatives discussed in \cite{Butcher_B2008,Hairer_B2008_I}.
\paragraph{Example} The tree $\BTreeVIIA$ corresponds to $f'f''(f,f)$. The trees of order 4 are 
$\TreeSet_4=\left\{\BTreeVA,\BTreeVIA,\BTreeVIIA,\BTreeVIIIA\right\}$,
$\alpha(\Ttree)=1$ for $\Ttree \in \TreeSet_4 \backslash
\left\{\BTreeVIA\right\}$, $\alpha\left(\BTreeVIA\right)=3$.
\begin{definition}[B-series\label{def:B-Series} \cite{Hairer_1974}] 
Let $a:\TreeSet \rightarrow \mathbb{R}$ be a mapping between the tree set
and real numbers. The following is called a B-series:
\begin{align}
\nonumber
B(a,y)&=a(\emptyset) y + \dt\, a(\BTreeIA) f(y) + \frac{\dt^2}{2} a(\BTreeIIA)
F( \BTreeIIA)(y) + \cdots \\
\label{eq:BSeries}
&= \sum_{\Ttree \in \TreeSet}
\frac{\dt^{\TreeOrder(\Ttree)} \alpha(\Ttree)}{\TreeOrder(\Ttree)!}
a(\Ttree) F(\Ttree) (y) \,,
\end{align} 
where $\TreeSet = \{\emptyset\} \bigcup \TreeSet_1 \bigcup \TreeSet_2
\bigcup \cdots$.
\end{definition}

The exact solution of an ODE system is a B-series
\cite{Hairer_1974}. Formally we have the following result.
\begin{theorem}[Exact solution as
    B-series\label{thm:B-Series}\cite{Hairer_1974}] The 
  exact solution of \eqref{eq:ODE}
satisfies 
\begin{align}
\nonumber
y^{(q)}(\Ttree) = \sum_{\Ttree \in \TreeSet} \alpha(\Ttree) F(\Ttree) (y) \,.
\end{align}
Therefore the exact solution is given by \eqref{eq:BSeries} with
$a(\Ttree)=1$, and the coefficient of $\dt^{\TreeOrder(\Ttree)}
F(\Ttree) (y)$ in the expansion is given by
$\frac{\alpha(\Ttree)}{\TreeOrder(\Ttree)!}$, $\forall \Ttree \in
\TreeSet_k$, $k=1,2,\dots,p$. 
\end{theorem}

The elementary weights in the expression of the B-series are
independent. The following result captures this aspect.
\begin{lemma}[Independence of elemetary
    differentials\label{lem:ED:Independence} \cite{Butcher_B2008}]  
The elementary differentials are independent. Moreover, the values of
the distinct elementary differentials for $(y^{\{j\}})'=\prod_{j=1}^{k}
(y^{\{j\}})^{m_j}/m_j!$, $y^{\{j\}}(\tx_0)=0$ are given by
$F(\Ttree_i)(y(\tx_0))=e_i$, where $k$ is the number of resulting trees
when the root is removed and $m_j$ is the number of copies of $\Ttree_j$.
\end{lemma}

The order of the numerical method can be defined in terms of a B-series as follows. 
\begin{definition}[Order of time-stepping methods\label{def:Order}] 
A numerical method applied to \eqref{eq:ODE} 
 with $f$ $p$-times
continuous differentiable is of order $p$ if the expansion of the
numerical solution satisfies \eqref{eq:BSeries} with
$\TreeOrder(\Ttree) \le p$.
\end{definition}
%

%
\subsection{Error equation\label{sec:Error:Equation}}
%

We now analyze the propagation of numerical errors through the time-stepping
processes. 

\begin{theorem}[Asymptotic expansion of the global errors \cite{Henrici_1962,Hairer_B2008_I}\label{thm:GlogalErrorExpansion}]
Suppose that method \eqref{eq:Henrici:OneStepNotation} possesses an
expansion \eqref{eq:local:error:expansion} under smoothness conditions
of Theorem \ref{th:glb:err:bnds}. Then the global error has an
asymptotic expansion of form 
\begin{align}
\label{eq:glb:expansion}
y(\tx)-y_\dt(\tx) = e_p(\tx) \dt^p + \dots + e_N(\tx) \dt^N +
E_\dt(\tx) \dt^{N+1}\,, 
\end{align}
where $E_\dt(\tx)$ is bounded on $\tx_0<\tx \le \Tx$ and $0 \le \dt \le
\dT$ for some $\dT$, and $e_p(\tx)$ satisfies 
\begin{align}
\label{eq:global:error:ode}
 e_{p}'(\tx) =
 \frac{\partial  f}{\partial y}(\tx,y)\cdot e_{p}(\tx) + d_{p+1}(\tx)\,,
 \quad e_p(\tx_0)=0\,.
\end{align}
The other $e_j(\tx)$ terms satisfy similar equations.  
\end{theorem}
\begin{proof}
Consider a perturbed method $\widehat{y}_{\dt}(\tx):=y_{\dt}(\tx) +
e_{p}(\tx) \dt^{p}$. Then $\widehat{y}_{\dt}(\tx)$ can be represented
as the numerical solution of a new method: $\widehat{y}_{n+1}=
\widehat{y}_{n} + \dt \widehat{\Phi}(\tx_{n},\widehat{y}_{n},\dt)$. 
By comparison with \eqref{eq:Henrici:OneStepNotation} we obtain
$\{\widehat{y}_{n}\}$, 
\begin{align}
\label{eq:local:error:PhiHat}
\widehat{\Phi}(\tx,\widehat{y}_{n} ,\dt) & =
\Phi(\tx_n,\widehat{y}_{n} -e_{p}(\tx_n) \dt^{p},\dt) + (e_{p}(\tx_n+\dt)
- e_{p}(\tx_n)) \dt^{p-1}\,.
\end{align}
Expanding the local error of the perturbed method with the Taylor
function defined by \eqref{eq:local:error:PhiHat} yields in general 
\begin{align}
\nonumber & y(\tx + \dt) - y(\tx) - \dt \widehat{\Phi}(\tx,y(\tx),\dt) 
\\& \qquad
\label{eq:Local:ErrorEquation}
 = \left( d_{p+1}(\tx)
 +\frac{\partial 
    f}{\partial y}(\tx,y)e_{p}(\tx)  - e_{p}'(\tx) \right)\dt^{p+1}
+{\mathcal O}(\dt^{p+2})  \,.
\end{align} 
We take $e_{p}(t)$ to satisfy \eqref{eq:global:error:ode}, so that by
Theorem  \ref{th:glb:err:bnds} it follows that
\begin{align}
\label{eq:global:error}
y(\tx) - y_{\dt}(\tx) = e_{p}(\tx) \dt^p + {\mathcal O}(\dt^{p+1})\,
\end{align}
determines the asymptotic expansion. For more details see
\cite{Hairer_B2008_I}. 
\end{proof}

Equations for the next terms in the global error expansion can be
obtained by using the same procedure; however, this is not pursued in
this study. 

%
\subsubsection{Estimating global errors using two methods\label{sec:two:methods}}
%

We now introduce the general global error estimation strategy used in
this study. This approach relies on propagating two solutions through
a linear time-stepping process that has the property of maintaining a
fixed ratio between the truncation error terms. This result will play
a crucial role in constructing the new methods discussed in Sec. \ref{sec:Methods:GEM} and can be
stated as follows.   
\begin{theorem}[Global error estimation with two methods\label{prop:two:methods}]
Consider numerical solutions $\{y_n\}$ and $\{\tilde{y}_n\}$ of
\eqref{eq:ODE} obtained by two time-stepping methods 
started from the same exact initial condition 
under the conditions
of Theorem \ref{thm:GlogalErrorExpansion}. If the local errors of the two
methods with increments $\Phi$ and $\tilde{\Phi}$ satisfy 
\begin{subequations}
\label{eq:local:error:ode:two:methods}
\begin{align}
\label{eq:local:error:ode:two:methods:1}
y(\tx + \dt) - y(\tx) - \dt \Phi(\tx,y(\tx),\dt)  & = d_{p+1}(\tx_{n}) \dt^{p+1}  + {\mathcal O} (\dt^{p+2})\,,\\
\label{eq:local:error:ode:two:methods:2}
y(\tx + \dt) - y(\tx) - \dt \tilde{\Phi}(\tx,y(\tx),\dt)  &= \gamma  d_{p+1}(\tx_{n}) \dt^{p+1}  + {\mathcal O} (\dt^{p+2})\,,
\end{align}
\end{subequations}
where $d_{p+1}(\tx_{n}) = \frac{1}{(p+1)!} \sum_{\Ttree \in
  \TreeSet_{p+1}}  \alpha(\Ttree) a(\Ttree) F(\Ttree)(y_{n})$  
with constant $\gamma \ne 1$, then the global
error can be estimated as  
\begin{align}
\label{eq:global:error:two:methods}
\varepsilon_{p}(\tx_{n}) = \frac{1}{1-\gamma} \left(\tilde{y}_{n}
- y_{n} \right) = e_p(\tx_{n}) \dt^{p} + {\mathcal O} (\dt^{p+1})=
\varepsilon(\tx_n) + O(\dt^{p+1})\,,  
\end{align} 
when $y_{0}=\tilde{y}_{0}=y(\tx_0)$; hence, $\varepsilon_n \asymp
y(t_n) - y_n$. 
\end{theorem}
\begin{proof}
Use \eqref{eq:glb:expansion} and \eqref{eq:global:error:ode} to
write the global error equations for the two methods with nearby
solutions: 
\begin{subequations}
\label{eq:global:error:ode:two:methods}
\begin{align}
\label{eq:global:error:ode:two:methods:1}
 e_{p}'(\tx) &=
 \frac{\partial  f}{\partial y}(\tx,y)\cdot e_{p}(\tx) + d_{p+1}(\tx)\,,
 \quad e_p(\tx_0)=0\,,\\
\label{eq:global:error:ode:two:methods:2}
 \tilde{e}_{p}'(\tx) &=
 \frac{\partial  f}{\partial y}(\tx,y)\cdot \tilde{e}_{p}(\tx) +
 \gamma d_{p+1}(\tx)\,,
 \quad \tilde{e}_p(\tx_0)=0\,.
\end{align}
\end{subequations}
It follows that the solutions of the two ordinary differential equations satisfy $ \gamma e_{p}(\tx)
=\tilde{e}_{p}(\tx)$. We can then verify
\eqref{eq:global:error:two:methods} by inserting
\eqref{eq:global:error}:
\begin{align}
\nonumber
\varepsilon_{p}(\tx_{n})  &=  \frac{1}{1-\gamma} \left(\tilde{y}_{n}  -
y_{n} \right)\\   
\nonumber
&= \frac{1}{1-\gamma} \left( y(\tx_n) - \tilde{e}_p(\tx_n) \dt^p  -
  y(\tx_) + e_p(\tx_n) \dt^p + {\mathcal O} (\dt^{p+1})\right) \\ 
\nonumber
&= e_p(\tx_n) \dt^p + {\mathcal O} (\dt^{p+1}) 
\end{align}
for $n=1,2, \dots$.
\end{proof}

We will refer to $\varepsilon$ as the exact global error and to
$\varepsilon_p$ as the order-$p$ global error estimate. A particular case is $\gamma=0$. Moreover, under the assumptions of Theorem \ref{prop:two:methods}, one can always compute a higher-order approximation by combining the two solutions.
\begin{corollary}
If $\gamma=0$ in Theorem \ref{prop:two:methods}, then we revert to the
case of using two methods of different orders, $p$ and $p+1$, to
estimate the global errors for the method of order $p$.
\end{corollary}
\begin{corollary}
A method of order $p+1$ can be obtained with conditions of Theorem
\ref{prop:two:methods} by
\begin{align}
\label{eq:high:order:estimate}
\widehat{y}_n = y_n + \varepsilon_p = \frac{1}{1-\gamma} \tilde{y}_n
- \frac{\gamma}{1-\gamma} y_n \,. 
\end{align}
\end{corollary}

We note that a related analysis has been carried out in
\cite{Shampine_1986b} with an emphasis of reusing standard codes for
solving ODEs with global error estimation. Moreover, Murua and
Makazaga take a similar approach at identifying the global error from
two related numerical solutions \cite{Murua_2000}. 

The result presented above is the basis for the developments in this
study. We introduce a new type of methods that provide {\em a posteriori}
error estimates, and we show that this procedure generalizes all
strategies that compute global errors by propagating multiple
solutions or integrating related problems. The validity of this approach
when variable time steps are used is discussed next.  

%
\subsubsection{Global errors with variable time
  steps\label{sec:variable:time:step}}
%
Following \cite{Hairer_B2008_I}, for variable time stepping we
consider $\tx_{n+1}-\tx_{n}=\nu(\tx_n) \dt$, $n=1,2,
\dots$ . 

Then the local error expansion \eqref{eq:local:error:expansion} becomes
\begin{align}
\nonumber
&y(\tx + \nu(\tx) \dt) - y(\tx) - \nu(\tx)  \dt \,\Phi(\tx,y(\tx),\dt) =\\
\nonumber
& \qquad 
 d_{p+1}(\tx)
 \nu(\tx)^{p+1}  \dt^{p+1} + \cdots  + d_{N+1}(\tx)
 \nu(\tx)^{N+1} \dt^{N+1} + {\mathcal O}(\dt^{N+2}) \,,
\end{align} 
and instead of \eqref{eq:local:error:PhiHat} we obtain
\begin{align}
\nonumber
&\widehat{\Phi}(\tx,\widehat{y}_{n} ,\nu(\tx)\dt)  =
\Phi(\tx,\widehat{y}_{n} -e_{p}(\tx) \dt^{p},\nu(\tx)\dt) + (e_{p}(\tx+\nu(\tx)\dt)
- e_{p}(\tx)) \frac{\dt^{p}}{\nu(\tx) \dt}\,.
\end{align}
Then \eqref{eq:Local:ErrorEquation} becomes
\begin{align}
\nonumber
&y(\tx +\nu(\tx) \dt) - y(\tx) - \nu(\tx)\dt \widehat{\Phi}(\tx,y(\tx),\nu(\tx)\dt)
 \\
\nonumber 
&
\qquad= \nu(\tx)\left( d_{p+1}(\tx) \nu(\tx)^p
 +\frac{\partial 
    f}{\partial y}(\tx,y)e_{p}(\tx)  - e_{p}'(\tx) \right)\dt^{p+1}
+{\mathcal O}(\dt^{p+2})  \,.
\end{align}
Instead of \eqref{eq:global:error:ode}, the global error $e_{p}(\tx)$ satisfies the following
equation:
\begin{align}
\label{eq:global:error:ode:VarT}
 e_{p}'(\tx) =
 \frac{\partial  f}{\partial y}(\tx,y)\cdot e_{p}(\tx) +  \nu(\tx)^p d_{p+1}(\tx)\,,
 \quad e_p(0)=0\,.
\end{align} 
The results introduced in this study and summarized by Theorem
\ref{prop:two:methods} carry over to variable time stepping with $\dt$
replaced by $\dt_{\rm max}=\max(\nu(\tx)\dt)$ and, therefore, allow
the application of such strategies in practical contexts.  

In this study we do not address the problem of time-step adaptivity based on global error estimates. In practice, the adaptivity can be based on asymptotically correct local error estimates that are provided directly by the methods proposed here.
%
\subsubsection{Methods satisfying the exact principal error equation} 
%
We next review a class of methods used for global error estimation. Consider an asymptotic error expansion in \eqref{eq:global:error:ode} of 
%
\begin{align}
\label{eq:global:error:ode:expansion}
 e(\tx) = \sum_{\Ttree \in  \TreeSet_p} \alpha(\Ttree) a(\Ttree)
 F(\Ttree)(y(\tx))\,, ~\tx>\tx_0\,, ~\textnormal{and} ~  e(\tx_0) = \sum_{\Ttree \in
   \TreeSet_p} \alpha(\Ttree) a(\Ttree)  F(\Ttree)(y(\tx_0)) \,,
\end{align} 
for some constant $a(\Ttree)$. By inserting
\eqref{eq:global:error:ode:expansion} in \eqref{eq:global:error:ode}
we obtain   
\begin{align}
\nonumber
d_{p+1}(\tx) =& \frac{d}{d \tx} \left[ \sum_{\Ttree \in  \TreeSet_p} \alpha(\Ttree)
  \mathbf{e}(\Ttree) F(\Ttree)(y(\tx)) \right] -
 \frac{\partial  f}{\partial y}(y(\tx)) \cdot  \sum_{\Ttree \in T_p}
 \alpha(\Ttree) \mathbf{e}(\Ttree) F(\Ttree)(y(\tx)) \\
\label{eq:global-local:error:ode:expansion}
       =& \sum_{\Ttree \in  \TreeSet_p} \alpha(\Ttree) \mathbf{e}(\Ttree) \left[
         \frac{d}{d \tx}  F(\Ttree)(y(\tx)) - \frac{\partial
           f}{\partial y}(y(\tx))   F(\Ttree)(y(\tx))
         \right] \,.
\end{align} 
%
This expression implies that if the local error satisfies
\eqref{eq:global-local:error:ode:expansion}, then
\eqref{eq:global:error:ode:expansion} is the exact solution of
\eqref{eq:global:error:ode}, and therefore the global errors can be
estimated directly, as described below.  

This strategy was indirectly introduced by Butcher
\cite{Butcher_P1969} in an attempt to break the order barriers of
multistage methods under the alias ``effective order.'' Stetter
\cite{Stetter_1971} observed the relationship between
\eqref{eq:global-local:error:ode:expansion} and the global
error \eqref{eq:global:error:ode}. This
strategy requires a starting procedure $\GLMS$ to enforce $e(\tx_0)$,
a method $\GLMM$ that satisfies
\eqref{eq:global-local:error:ode:expansion}, and a finalizing procedure
$\GLMF$ to extract the global error. We denote by
$\GLMS(\circ)$,  $\GLMM(\circ)$,  $\GLMF(\circ)$ the application of
each method on solution $\circ$. Stetter \cite{Stetter_1971} found that $\GLMS$
  and $\GLMF$ can be one order less than  $\GLMM$. Examples of such triplets can
be found in many studies
\cite{Butcher_P1969,Stetter_1971,Scholz_1989,Richert_1987,Merluzzi_1978,Prince_1978,Rive_1975}.%

\begin{boxedminipage}{0.9\textwidth}
Algorithm [A:ExPrErEq]: Methods with exact principal error equation \cite{Stetter_1971} \\Solve
\begin{subequations}
\label{eq:ExactPrincipalError}
\begin{align}
&y_1=\GLMS(y_0)\,, & y(\tx_0)=y_0\\
&\left\{
\begin{array}{l}
y_n=\GLMM(y_{n-1})\\
\varepsilon_n=y_n-\GLMF(y_{n-1})
\end{array}
\right. ~~ n=2,\,3,\,\dots\,, & \textnormal{so~that}~\eqref{eq:global-local:error:ode:expansion}.
\end{align} 
\end{subequations}  
\end{boxedminipage}

One such scheme is provided in Appendix \ref{sec:ExactPrincipalError:Prince_1978}.
However, a caveat is that methods based on explicit Runge-Kutta schemes
require as many nonzero stage coefficients as the order of the method
because $\GLMM$ needs to have a nonzero tall tree of $p+1$, hence, the
effective order is limited by $p \le s$. For instance, an order 5 method requires
at least five stages. This requirement comes from the fact that tall
trees need to be nonzero in
\eqref{eq:global-local:error:ode:expansion}. However, this strategy is
still effective for high orders. Recently the effective order was discussed in
\cite{Butcher_1999,Butcher_1998a,Butcher_1998b,Butcher_1998c,Butcher_1997c,Hadjimichael_2013}. Effective order through method composition has recently been discussed in \cite{Chang_2008}. 

Although this concept is attractive in terms of efficiency, Prince and Wright
\cite{Prince_1978} noted a severe problem with using it for
global error estimation: If the system has unstable components, then
the error approximation becomes unreliable, as can be seen in
Fig. \ref{fig:nr:conv:principal:error} discussed later. This is a
severe limitation because having unstable components makes the local
error estimates unreliable, and this is precisely the case when one
would need to use global error estimation.

%
\subsection{Differential correction\label{sec:differential:correction}}
%

The differential correction techniques for global error estimation are based on the work of
Zadunaisky \cite{Zadunaisky_1976} and Skeel \cite{Skeel_1976}. The discussion of these procedures is
deferred to Sec. \ref{sec:solving:for:correction}.

%
\subsubsection{Error equation and the defect}
%

We follow the exposition in \cite{Lang_2007,Zadunaisky_1976} and assume that there exists
a solution $z(\tx)$ of a perturbed system 
\begin{align}
&z(\tx)'=f(\tx,z(\tx))-r(\tx)\,;~ z(\tx_0)=z_0\,,~ r(0) = y_0-z_0 \,,~
  \tx_0<\tx \le \Tx \label{eq:DC:ODE:pert}\,,
\end{align}
close to $y(\tx)$. The error function (between the solutions of \eqref{eq:ODE} and \eqref{eq:DC:ODE:pert}) is given by \cite{Lang_2007}
\begin{align}
e(\tx) &= y(\tx) - z(\tx)\,, \label{eq:DC:ErrorEQ}\\
e'(\tx) &= A(\tx)  e(\tx) - r(\tx)\,,\qquad A= \int_{0}^{1}
f'(\tx,y(\tx)  + s e(\tx) ) \: ds\,. 
 \label{eq:DC:ODE:ErrorEQ} 
\end{align}

If $e(\tx_0)=0$ and we approximate $A(\tx) = \frac{\partial f}{\partial
  \tx}(\tx,y) + {\mathcal O}(e(\tx))$ in \eqref{eq:DC:ODE:ErrorEQ}, then we
obtain 
\begin{align}
 e' (\tx)&=   \frac{\partial f}{\partial
  \tx}(\tx,y)  e(\tx) - r(\tx) \,, ~ e(\tx_0)=0 \,,~\tx_0<\tx \le
  \Tx
 \label{eq:DC:ODE:ErrorEQ:2} \,,
\end{align}
with $r(\tx)=-d_{p+1}(t) \dt^p$. This is asymptotically
equivalent to solving the first variation (leading term) of the global
error equation for $e_p$; i.e., \eqref{eq:global:error:ode}. Consider now that the
nearby solution, $z(\tx)$, is obtained through 
an interpolatory function $P(\tx)$, and define the defect $D(\tx)$ as
\begin{align}
\label{eq:DC:Defect:Function}
D(\tx)=f(\tx,P(\tx)) - P'(\tx)\,.
\end{align}
Estimates of the local truncation errors can be obtained by using continuous
output \cite{Enright_2000}. Lang and Verwer \cite{Lang_2007} showed
that if $P(\tx)$ is obtained through Hermite interpolation, then 
\begin{align}
\nonumber
D(\tx)= [y'(\tx) - f(\tx,y(\tx))] - [f(\tx,P(\tx)) - P'(\tx)] = {\mathcal
  O}(\dt^3)\,,~~ \tx \in (\tx_{n},\tx_{n+1}) \,,
\end{align}
and in particular $D(\tx_n+\frac{\dt}{2}) = {\mathcal
  O}(\dt^4)$. Furthermore, a relation between the defect at
$\tx_n+\frac{\dt}{2}$
and the leading term of the local truncation error,
$D(\tx_{n+\frac{1}{2}}) = \frac{3}{2} d_{p+1}(\tx_n)\dt + {\mathcal
  O}(\dt^{p+1})\,, ~~ 1 \le p \le 3$, can be obtained.
%
We can then set $r(\tx) = \frac{2}{3} D(\tx_{n+\frac{1}{2}})$, $\tx
\in (\tx_{n},\tx_{n+1})$, and \eqref{eq:global:error:ode} and
\eqref{eq:DC:ODE:ErrorEQ:2} become
\begin{align}
 e' (\tx)&= f'(\tx_n,y_n)  e(\tx) - r(\tx_{n+\frac{1}{2}}) \,, ~ e(\tx_0)=0
 \,,~\tx_n < \tx \le \tx_{n+1}\,,~~ n=0,1,\dots,N\,.
 \label{eq:DC:ODE:ErrorEQ:LangVerwer} 
\end{align}
%

%
\subsubsection{Solving the error equation}
%
If the Jacobian of $f$ is available, then
\eqref{eq:DC:ODE:ErrorEQ:LangVerwer}  can be solved directly as in
\cite{Lang_2007}. 

\begin{boxedminipage}{0.9\textwidth}
%
Algorithm [A:SoErEq]: Solving the error equation \cite{Lang_2007} \\Solve
%
\begin{subequations}
\label{eq:GE:LangVerwer}
\begin{align}
\label{eq:GE:LangVerwer:ODE}
y' &= f(t,y)\,, & y(\tx_0)=y_0\\
\label{eq:GE:LangVerwer:ERROR}
\varepsilon' (\tx)&= J \varepsilon(\tx) + [d_{p+1}(\tx_n) \dt]  \,, &  \varepsilon(\tx_0)=0\\
\nonumber
d_{p+1}(\tx_n) &= y(\tx_{n+1}) - y_{n+1} + {\mathcal O}(\dt^{p+2})\,,
~ J=\frac{\partial f}{\partial y}\,.
\end{align}
\end{subequations}
%
\end{boxedminipage}

The authors of \cite{Lang_2007} argue that
\eqref{eq:GE:LangVerwer:ERROR} can be solved with a cheaper, lower-order
method. In this case, however, the bulk of the work resides on
determining $d_{p+1}$, which can be estimated by following the steps
discussed in the previous section. 

%
\subsubsection{Solving for the correction\label{sec:solving:for:correction}}
%

This  approach follows the developments presented in
\cite{Zadunaisky_1976,Skeel_1986,Peterson_B1986} and further refined
in \cite{Dormand_1984,Dormand_1985,Dormand_1989,Dormand_1994}. We start from \eqref{eq:DC:ODE:pert} and denote by $P(\tx)$ its exact
solution. Equation \eqref{eq:DC:ErrorEQ} becomes
\begin{subequations}
\begin{align}
e(\tx)  &= y(\tx) -  P(\tx) \,, ~\textnormal{and}\label{eq:SC:ErrorEQ}\\
 e'(\tx)&=(y(\tx)-P(\tx))' = f(\tx,y(\tx)) - P(\tx)'= f(\tx,P(\tx)+e(\tx))  - P'(\tx) 
 \label{eq:DC:ODE:ErrorEQ:2:bis} \,.
\end{align}
\end{subequations}
We can see the connection between
\eqref{eq:DC:ODE:ErrorEQ:2:bis} and \eqref{eq:global:error:ode} by starting with \eqref{eq:global:error:ode}: 
\begin{align}
e' (\tx) &= f'(\tx,y(\tx)) e(\tx) + D(\tx) =\nonumber
f(\tx,y(\tx)) - f(\tx,y(\tx) - e(\tx))  + f(\tx,P(\tx)) - P'(\tx) +
{\mathcal O}(e^2) \\
&\approx f(\tx,y(\tx)) - P'(\tx) = f(\tx,P(\tx)+e(\tx)) - P'(\tx) \,, \nonumber
\end{align}
where we neglected the higher-order terms and used
\eqref{eq:DC:Defect:Function} and \eqref{eq:SC:ErrorEQ}. 
The equations to be solved are known as the procedure for solving for
the correction \cite{Skeel_1986}.

\begin{boxedminipage}{0.9\textwidth}
%
Algorithm [A:SolCor]: Solving for the correction\cite{Skeel_1986} \\Solve
\begin{subequations}
\label{eq:Solving:Correction:ODE}
\begin{align}
\label{eq:Solving:Correction:ODE:y}
y' &= f(t,y)\,, &  y(\tx_0)=y_0 \\
\label{eq:Solving:Correction:ODE:z}
\varepsilon' &=f(\tx,P(\tx) + \varepsilon) - P'(\tx) \,, &
\varepsilon(\tx_0)= 0\\
\nonumber
P(\tx) &\approx y(\tx) - y_{\dt}(\tx)\,. 
\end{align}
\end{subequations}
%
%
\end{boxedminipage}

We will show that equations \eqref{eq:Solving:Correction:ODE} (in
[A:SolCor]) can be solved by using a
general linear method representation \eqref{eq:Solving:Correction:GLM}
described in Sec. \ref{sec:Zadunaisky}.

The related 
Zadunaisky procedure
\cite{Zadunaisky_1976} is as follows. Calculate the polynomial of
order $p$,  $P(\tx)$, by using Lagrange interpolation ${\mathcal
  L}_p(y_{\dt}(\tx))$ over several steps and then apply a similar
procedure as in \eqref{eq:Solving:Correction:ODE} on a perturbed system.

\begin{boxedminipage}{0.9\textwidth}
%
Algorithm [A:ZaPr]: Zadunaisky procedure\cite{Zadunaisky_1976} \\Solve
\begin{subequations}
\label{eq:Zadunaisky:ODE}
\begin{align}
\label{eq:Zadunaisky:ODE:y}
y' &= f(t,y)\,, &  y(\tx_0)=y_0 \\
\label{eq:Zadunaisky:ODE:z}
z' &=f(\tx,P(\tx)) - P'(\tx) - f(\tx,z) \,, &
z(\tx_0)= y_0\\
\nonumber
\varepsilon_n &= z_n-y_n & P(\tx) = {\mathcal L}_p(y_{\dt}(\tx))
\,. 
\end{align}
\end{subequations}
%
%
\end{boxedminipage}
%

%
\subsection{Extrapolation approach\label{sec:extrapolation}}
%

The global error estimation through extrapolation dates back to
\cite{Richardson_1927}. The
procedure is the following. Propagate two 
solutions $y_{\dt,n}$ and $y_{\frac{\dt}{2},n}$, one with $\dt$ and one with $\dt/2$, each with global errors
$\varepsilon_{\dt,n} = y(t_n) - y_{\dt,n} $,
$\varepsilon_{\frac{\dt}{2},n} = y(t_n) - y_{\frac{\dt}{2},n} $,
respectively. 

Then it follows that  by
using a method of order $p$ one obtains \cite{Henrici_1962}
\begin{subequations}
\label{eq:glabal:error:extrapolation}
\begin{align}
\varepsilon_{\dt,n} &= y(t_n) - y_{\dt,n}  = e_p \dt^p + {\mathcal
  O}(\dt^{p+1}) \,,
\nonumber
\\ 
\varepsilon_{\frac{\dt}{2},n} &=  y(t_n) - y_{\frac{\dt}{2},n}  = e_p
\left(\frac{\dt}{2}\right)^p + {\mathcal O}(\dt^{p+1}) \,.
\nonumber
\end{align}
The global error and a solution of one order higher can be obtained as 
\begin{align}
\varepsilon_{\dt,n}&= \frac{2^p}{1-2^p} (y_{\dt,n} - y_{\frac{\dt}{2},n}) + {\mathcal
  O}(\dt^{p+1}) \label{eq:glabal:error:extrapoation:err} \,,\\
\widehat{y}_{\dt,n} &= y_{\dt,n} + \varepsilon_{\dt,n} =
\frac{1}{1-2^p} y_{\dt,n} -  \frac{2^p}{1-2^p} y_{\frac{\dt}{2},n}
= y(\tx_n) + {\mathcal O}(\dt^{p+1}) 
\label{eq:glabal:error:extrapoation:higher:order}
\,.
\end{align}
\end{subequations}
These statements are a particular instantiation of
\eqref{eq:global:error:two:methods} and \eqref{eq:high:order:estimate}
with $\gamma=1/2^p$.

\noindent\begin{boxedminipage}{0.99\textwidth}
%
Algorithm [A:Ex]: Extrapolation \\Solve $y' = f(t,y)$  by using a method of order $p$ with two time
steps $\dt$ and $\dt/2$
%
\begin{subequations}
\label{eq:GE:Extrap}
\begin{align}
\label{eq:GE:Extrap:ODE}
y' &= f(t,y) \Rightarrow y_{\dt,n}\,,~ y_{\frac{\dt}{2},n} \,, & y(\tx_0)=y_0\\
\label{eq:GE:Extrap:ERROR}
\varepsilon&=\frac{2^p}{1-2^p} (y_{\dt,n} - y_{\frac{\dt}{2},n})\,.
\end{align}
\end{subequations}
%
%
\end{boxedminipage}

%
\subsection{Underlying higher-order method\label{sec:Under:High:Order}}
%

All the methods described in this study attempt to use an
  underlying higher-order method to estimate the global error. In the
case of [A:ExPrErEq] the exact principal error algorithm
\eqref{eq:ExactPrincipalError} and of [A:SoErEq] solving the error equation
\eqref{eq:GE:LangVerwer}, we find that the actual equation being solved
is modified to include the truncation error 
term. By adding \eqref{eq:GE:LangVerwer:ODE} and
\eqref{eq:GE:LangVerwer:ERROR} one obtains
\begin{align*}
y' + \varepsilon' = \widehat{y}' &= f(y) + J \varepsilon + D(y)\\
\widehat{y}' &= f(\widehat{y}-\varepsilon) + J \varepsilon + D(y)\\
\widehat{y}'&= f(\widehat{y}) + D(y)\,.
\end{align*}

In the case of the Zadunaisky algorithm [A:SolCor] \eqref{eq:Solving:Correction:ODE},
one can recover the underlying higher-order method by replacing the
error term in \eqref{eq:Solving:Correction:ODE:z} with 
$\widehat{y}$ from \eqref{eq:high:order:estimate}  and using  the
conditions imposed on $P$ (see \cite{Dormand_1984}). We show an
example  in Sec. \ref{sec:Zadunaisky}. The extrapolation algorithm [A:Ex]
\eqref{eq:GE:Extrap} reveals the higher-order estimate directly in
\eqref{eq:glabal:error:extrapoation:higher:order}.  

%
\section{General linear methods\label{sec:GLMs}}
%

The methods introduced in this study are represented by GL
schemes. In this context, we take advantage of the existing theory
underlying GL methods and augment it with global error estimation
capabilities. Two key elements are required. The first is a
consequence of Theorem \ref{prop:two:methods}, which imposes a
restriction on the truncation error. The second has to do with restricting the
interaction between the two outputs. Both will be addressed in Sec. 
\ref{sec:Methods:GEM}. The advantage of representing
existing global error strategies in a more general framework is that
it allows for the development of more robust methods. For this we need
GL methods that carry at least two quantities as discussed above. In
this section we present the GL methods theory without built-in error
estimates. 

General linear methods were introduced by Burrage and Butcher 
\cite{Burrage_1980}; however, many GL-type schemes have been proposed
to extend either Runge-Kutta methods \cite{Gragg_1964} to linear
multistep (LM) or vice versa \cite{Butcher_1965,Gear_1965}, as well as
other extensions
\cite{Butcher_1966,Cooper_1978,Hairer_1973,Skeel_1976} or directly as
peer methods \cite{Schmitt_2004,Podhaisky_2005}. GL methods are 
 thus a generalization of both RK and LM methods, and we use the GL
formalism to introduce new methods that provide asymptotically correct
global error estimates.

Denote the solution at the current step ($n-1$) by an $r$-component vector 
$\mathbbm{y}^{[n-1]}= [\mathbbm{y}_{(1)}^{[n-1]} \, \mathbbm{y}_{(2)}^{[n-1]}\dots 
\mathbbm{y}_{(r)}^{[n-1]}]^T$, which contains the available information in the form 
of numerical approximations to the ODE (\ref{eq:ODE}) solutions and their derivatives 
at different time indexes. To increase clarity, we henceforth denote
the time index inside square brackets. The stage values (at step $n$) are denoted by 
${\bf Y}_{(i)}$ and stage derivatives by $\mathbf{f}_{(i)}=f\left(\mathbf{Y}_{(i)}\right)$, 
$i=1,\,2,\,\dots,\,s$, and can be compactly represented as
$\mathbf{Y} = \left[\mathbf{Y}_{(1)}^T\, \mathbf{Y}_{(2)}^T\dots
  \mathbf{Y}_{(s)}^T\right]^T$ and  $\mathbf{f} = \left[\mathbf{f}_{(1)}^T
  \, \mathbf{f}_{(2)}^T \dots \mathbf{f}_{(s)}^T\right]^T$.  

The $r$-value $s$-stage GL method is described by
%
\begin{eqnarray}
\label{eq:GLM}
\begin{array}{rl}
\displaystyle \mathbf{Y}_{(i)} = & \displaystyle \dt \sum_{j=1}^{s}\GLMA_{ij}
 \mathbf{f}_{(j)} + \sum_{j=1}^{r} \GLMU_{ij} \mathbbm{y}_{(j)}^{[n-1]} 
\,,~ i=1,\,2,\,\dots,\,s\,, \\
\displaystyle \mathbbm{y}_{(i)}^{[n]}=& \displaystyle
 \dt \sum_{j=1}^{s}\GLMB_{ij}
 \mathbf{f}_{(j)} + \sum_{j=1}^{r} \GLMV_{ij} \mathbbm{y}_{(j)}^{[n-1]}  \,,~
i=1,\,2,\,\dots,\,r\,, 
\end{array}
\end{eqnarray}
%
where ($\GLMA, \GLMU, \GLMB, \GLMV$) are the 
coefficients that define each method and can be grouped further into
the GL matrix $\GLMM$:  
%
\begin{eqnarray}
\nonumber
\left[
\begin{array}{c}
\mathbf{Y}\\
\mathbbm{y}^{[n]}
\end{array}
\right]
=
\left[
\begin{array}{cc}
\GLMA \otimes I_m&\GLMU\otimes I_m\\
\GLMB\otimes I_m&\GLMV\otimes I_m
\end{array}
\right]
\left[
\begin{array}{c}
\dt \mathbf{f}\\
\mathbbm{y}^{[n-1]}
\end{array}
\right]
=
\GLMM
\left[
\begin{array}{c}
\dt\mathbf{f}\\
\mathbbm{y}^{[n-1]}
\end{array}
\right]\,.
\end{eqnarray}
%
Expression (\ref{eq:GLM}) is the most generic representation of GL
methods \cite[p. 434]{Hairer_B2008_I} and encompasses both RK methods 
($r=1$, $s>1$) and LM methods ($r>1$, $s=1$) as particular cases. In
this work we consider methods with $r=2$, where the first component
represents the primary solution of the problem
\eqref{eq:local:error:ode:two:methods:1} and the second component can
represent either the defect \eqref{eq:global:error:two:methods} or the
secondary component \eqref{eq:local:error:ode:two:methods:2}. Only
multistage-like methods are considered; however, multistep-multistage
methods ($r>2$) are also possible. 

If method (\ref{eq:GLM}) is consistent (there exist vectors $q_0$,
$q_1$ such that $\GLMV q_0=q_0$, $\GLMU q_0=\mathbbm{1}$, and
$\GLMB \mathbbm{1}+\GLMV q_1=q_0+q_1$ \cite[Def. 3.2 and
3.3]{Butcher_2006}) and stable ($\Vert \GLMV^n \Vert$ remains
bounded, $\forall n = 1,2,\dots$ \cite[Def. 3.1]{Butcher_2006}), then the
method (\ref{eq:GLM}) is convergent \cite[Thm. 3.5]{Butcher_2006},
\cite{Butcher_B2008,Jackiewicz_B2009}.  In-depth descriptions and survey materials on GL methods can be found in
\cite{Butcher_2001a,Butcher_2006,Butcher_B2008,Hairer_B2008_I,Jackiewicz_B2009}. 
In this study we use self-starting methods that do not require a
solution history and specialized starting procedures. In general the initial input vector $\mathbbm{y}^{[0]}$ can
be generated through a starting procedure, $\GLMS=\left\{S_i:
\mathbb{R}^m \rightarrow \mathbb{R}^m \right\}_{i=1 \dots r} $,
represented by generalized RK methods; see
\cite[Chap. 53]{Butcher_B2008} and \cite{Constantinescu_A2010b}. The
final solution is typically obtained by applying a ``finishing
procedure,'' $\GLMF=\left\{F_i:
\mathbb{R}^m \rightarrow \mathbb{R}^m \right\}_{i=1 \dots r} $, to the
last output vector, in our case this is also trivial. We denote
by the GL process the GL method applied $n$ times and described by
$\GLMS\GLMM^n \GLMF$; that is, $\GLMM$ is applied $n$ times on the
vector provided by $\GLMS$, and then $\GLMF$ is used to 
extract the final solution.

\subsection{Order conditions for GL methods\label{sec:GL:order}}
The order conditions rely on the theory outlined by Butcher et al. 
\cite{Butcher_B2008,Constantinescu_A2009b,Constantinescu_A2010b}. The
derivatives of the numerical and exact 
solution are represented by rooted trees and expressed as a B-series
\cite{Butcher_1972,Hairer_1974} as delineated in Theorem
\ref{thm:B-Series} and order definition \ref{def:Order}. We use an
algebraic criterion to characterize the order conditions for GL methods as follows. 
Let $\Ttree \in \TreeSet$ and $E^{(\theta)}:\TreeSet \rightarrow \mathbb{R}$,
the ``exact solution operator'' of differential equation
(\ref{eq:ODE}), which represents the \emph{elementary weights for the
  exact solution} at $\theta \dt$. If 
$\theta=1$, then
$E^{(1)}(\Ttree)=E(\Ttree)=(\sigma(\Ttree)
\alpha(\Ttree))/\TreeOrder(\Ttree)!$ and
$E\left(\left\{\BTreeIA,\BTreeIIA,\BTreeIIIA,\BTreeIVA\right\}\right)=\{1,1/2,1/3,1/6\}$
for $\TreeOrder(\Ttree) \le 3$. The order can be analyzed
algebraically by introducing a  
mapping $\xi_i:\TreeSet \rightarrow \mathbb{R}$, $\xi_i(\emptyset)=1$
in our case,
$\xi_i(\Ttree)=\Phi^{(i)}(\Ttree)$, where $\Phi^{(i)}(\Ttree)$,
$i=1,\dots, r$, results from the starting procedure and $\emptyset$
represents the ``empty tree.'' Then for the general linear method
$(\GLMA,\GLMU,\GLMB,\GLMV)$, one has 
\begin{align}
\label{eq:GLM:error}
\eta(\Ttree)=\GLMA \eta D(\Ttree) + \GLMU \xi(\Ttree)\,, \qquad
\widehat{\xi}(\Ttree) = \GLMB \eta D(\Ttree) + \GLMV \xi(\Ttree)\,,
~~\Ttree \in   \TreeSet\,,
\end{align}
where $\eta$, $\eta D$ are mappings from $\TreeSet$ to scalars that
correspond to the internal stages and stage derivatives and
$\widehat{\xi}$ represents the output vector. The exact weights are
obtained from $[E\xi](\Ttree)$. The order of the GL method can be
determined by a direct comparison between $\widehat{\xi}(\Ttree)$ and
$[E\xi](\Ttree)$. More details can be found in \cite{Butcher_B2008},
where a criterion for order $p$ is given for a GL method described by
$\GLMM$ and $\GLMS$.  Therefore, in general, an order $p$ GL method results from the direct comparison of elementary
wights of $[\GLMM^n](\Ttree_j)=[E^n\xi](\Ttree_j)$ $\forall
\Ttree_j,\, \TreeOrder(\Ttree_j) \le p$. This criterion is a direct consequence of
\cite[Def. 3 and Prop. 1]{Constantinescu_A2009b}. In our particular
case, methods satisfying Theorem \ref{prop:two:methods} can be
developed by enforcing  \eqref{eq:local:error:ode:two:methods} on the
corresponding solution vector. We further elaborate the order
conditions in our particular case in Sec. \ref{sec:Order:GEE}.

\subsection{\label{sec:LS}Linear stability of GL methods}
The linear stability analysis of method (\ref{eq:GLM}) is performed on
a linear scalar test problem: $y'(t) = a y(t)$, $a
\in\mathbb{C}$. Applying (\ref{eq:GLM}) to the
test problem yields a solution of form $y^{[n+1]}= R(z)\,y^{[n]}$, 
%
\begin{align}
\label{eq:GL:StabMatrix}
R(z)&=\GLMV+z \GLMB \left(I_s-z \GLMA \right)^{-1}
\GLMU\,,\\
\label{eq:GL:StabFunction}
\Phi(w,z) &= \det(w I_r - R(z))\,,
\end{align}
%
where $z=a \dt$, $R(z)$ is referred to as the stability matrix of
the scheme, and $\Phi(w,z)$ is the stability function.

For given $z$, method (\ref{eq:GLM}) is linearly stable if the
spectral radius of $R(z)$ is contained by the complex unit disk. The
stability region is defined as the set ${\mathcal S}=\left\{z \in
  \mathbb{C}: |R(z)| \le 1 \right\}$. The linear stability region
provides valuable insight into the method's behavior with nonlinear
systems. Additional details can be found
in \cite{Butcher_B2008}.  

%
\section{Methods with global error estimation (GEE)\label{sec:Methods:GEM}}
%
We now introduce GL methods with global and local error estimation. We
focus on Runge-Kutta-like schemes in the sense that the resulting GL
methods are self-starting multistage schemes. We therefore restrict our exposition
to methods that carry two solutions explicitly and where $r=2$. Generalizations
are possible but not addressed here. The methods are given
in two forms that use different input and output quantities. The first
form used for numerical analysis results in a scheme denoted by
$\GLyy$ that evolves two solutions of the ODE problem $y$ and
$\tilde{y}$. Methods $\GLyy$ take the following form:
%
%
\begin{eqnarray}
\label{eq:GLM:Global}
\begin{array}{rl}
\displaystyle Y_{(i)} = &\displaystyle \dt \sum_{j=1}^{s}\GLMA_{ij}
 f(Y_{(j)}) +  \GLMU_{i,1} y_{(1)}^{[n-1]}  +  \GLMU_{i,2} y_{(2)}^{[n-1]} 
\,,~ i=1,\,2,\,\dots,\,s\,, \\
\displaystyle y_{(1)}^{[n]}=& \displaystyle
\dt  \sum_{j=1}^{s} \GLMB_{1,j}
 f(Y_{(j)}) +  \GLMV_{1,1} y_{(1)}^{[n-1]}  +  \GLMV_{1,2} y_{(2)}^{[n-1]} \,,~\\
\displaystyle y_{(2)}^{[n]}=& \displaystyle
\dt \sum_{j=1}^{s} \GLMB_{2,j}
 f(Y_{(j)}) +  \GLMV_{2,1} y_{(1)}^{[n-1]}  +  \GLMV_{2,2} y_{(2)}^{[n-1]}   \,.~
\end{array}
\end{eqnarray}
%
%
We will consider $V=I_r$, although more general forms can also be considered. The 
second form, denoted by $\GLye$, is given as a method that evolves the
solution of the base method and the error explicitly,
$y$ and $\varepsilon$, as $\{y^{[n]},\varepsilon^{[n]}\} =
\GLye(\{y^{[n-1]},\varepsilon^{[n-1]}\})$, and has a more practical 
flavor. Both forms can be expressed
as GL methods with tableaux $(\GLMA_{y\tilde{y}},\GLMU_{y\tilde{y}},\GLMB_{y\tilde{y}},\GLMV_{y\tilde{y}})$ and
$(\GLMA_{y\varepsilon},\GLMU_{y\varepsilon},\GLMB_{y\varepsilon},\GLMV_{y\varepsilon})$,
respectively; and one can switch between the forms as explained below. 

\begin{lemma}\label{lem:GLMyy:GLMye}
GL methods of form \eqref{eq:GLM:Global} that satisfy the conditions
of Theorem \ref{prop:two:methods} with coefficients
$(\GLMA_{y\tilde{y}}, \allowbreak\GLMU_{y\tilde{y}}, \allowbreak\GLMB_{y\tilde{y}},\allowbreak \GLMV_{y\tilde{y}})$,
where $\mathbbm{y}^{[n]}=[(y^{[n]})^T,(\tilde{y}^{[n]})^T]^T$, and
$(\GLMA_{y\varepsilon},\GLMU_{y\varepsilon},\GLMB_{y\varepsilon},\GLMV_{y\varepsilon})$,
where $\mathbbm{y}^{[n]}={[(y^{[n]})^T,(\varepsilon^{[n]})^T]}^T$, are
related by  
\begin{align}
\label{eq:GLyy:to:GLye}
&\GLMA_{y\tilde{y}}=\GLMA_{y\varepsilon} \,,~ \GLMV_{y\tilde{y}}=\GLMV_{y\varepsilon}\,,
~\GLMU_{y\tilde{y}}=\GLMU_{y\varepsilon}  T_{y\varepsilon}^{-1} \,,~
  \GLMB_{y\tilde{y}}(1,:)= T_{y\varepsilon} \GLMB_{y\varepsilon}\,,
\end{align}
where $T_{y\varepsilon}=\left[ \begin{array}{cc} 1&0 \\1 & 1-\gamma \end{array} \right]$.
\end{lemma}
\begin{proof}
We start with a $\GLye$ method defined by
$(\GLMA_{y\varepsilon},\GLMU_{y\varepsilon},\GLMB_{y\varepsilon},\GLMV_{y\varepsilon})$
and write the resulting expression by applying \eqref{eq:GLM:Global}
with $y_{(1)}^{[n]}=y^{[n]}$ and $y_{(2)}^{[n]}=\varepsilon^{[n]}$. We then
replace $\varepsilon^{[n]}$ with $\frac{1}{1-\gamma}
\left(\tilde{y}^{[n]} - y^{[n]} \right)$ as in Theorem
\ref{prop:two:methods}, \eqref{eq:global:error:two:methods}. The
resulting expression can then be written as a $\GLyy$ scheme with
$y_{(1)}^{[n]}=y^{[n]}$ and $y_{(2)}^{[n]}=\tilde{y}^{[n]}$. This calculation
leads to \eqref{eq:GLyy:to:GLye}. This transformation is unique as
long as $\gamma \ne 1$.
\end{proof}

The following algorithm is proposed.

\noindent\begin{boxedminipage}{0.97\textwidth}
Algorithm [A:GLMGEE]: General linear methods with global error estimation \\
Initialize: $y^{[0]}=y(\tx_0)=y_0$, $\varepsilon^{[0]}=\varepsilon(\tx_0)=0$.\\
\noindent{} Solve: $y' = f(t,y)$ using  
%
%
\begin{subequations}
\label{eq:GE:GL}
\begin{align}
\label{eq:GE:GL:a}
\{y^{[n]},\varepsilon^{[n]}\} =&
\GLye(\{y^{[n-1]},\varepsilon^{[n-1]}\})\,,&[\textnormal{solution,~GEE}]\\
\label{eq:GE:GL:local:error}
\varepsilon_{\rm loc}=&\varepsilon^{[n]}-\varepsilon^{[n-1]}\,,&
           [\textnormal{local~error}]\\
\label{eq:GE:GL:high:order}
\widehat{y}^{[n]}=& y^{[n]} + \varepsilon^{[n]} = \frac{1}{1-\gamma} \tilde{y}^{[n]}
- \frac{\gamma}{1-\gamma} y^{[n]}\,.&[\textnormal{high~order}]
\end{align}
\end{subequations}
\end{boxedminipage}

We note that the algorithm above is suitable for self-starting GL
methods using fixed and adaptive time steps (see
Sec. \ref{sec:variable:time:step} and results in Fig. \ref{fig:nr:err:adapt}). For methods that are not
self-starting, a finalizing procedure might be necessary after each
step to extract the error components in \eqref{eq:GE:GL:a}; however,
this aspect is not addressed in this study. Moreover, as expected the
cumulative sum of the local errors yields 
the global error as suggested by \eqref{eq:GE:GL:local:error} and
illustrated through results in Fig. \ref{fig:nr:err:adapt}. 

%
\subsection{Consistency and preconsistency analysis}
%

We now discuss consistency and preconsistency conditions in the case
of a method with $r=2$. Following \cite{Jackiewicz_B2009}, we require that
\begin{subequations} 
\label{eq:prec}
\begin{align}
\label{eq:prec:assump:in}
y^{[n-1]}_i=&q_{i,0} y(\tx_{n-1}) + \dt q_{i,1} y'(\tx_{n-1}) + {\mathcal
  O}(\dt^2)\,, \quad i=1,2\\
\label{eq:prec:stages}
Y_i=&y(\tx_{n-1} + c_i \dt) + {\mathcal O}(\dt^2) \,, \quad i=1,2,\dots,
s\\
\label{eq:prec:assump:out}
y^{[n]}_i=&q_{i,0} y(\tx_{n}) + \dt q_{i,1} y'(\tx_{n}) + {\mathcal
  O}(\dt^2)\,, \quad i=1,2\,.
\end{align}
\end{subequations}
From \eqref{eq:prec:stages} we obtain 
\begin{align*}
y(\tx_{n-1}) \qquad     \qquad          &= (u_{i,1} q_{1,0} + u_{i,2} q_{2,0} )y(\tx_{n-1}) \\
 \qquad + c_i \dt y'(\tx_{n-1}) & ~~  +
 \dt (u_{i,1}  q_{1,1} +  u_{i,2} q_{2,1}) y'(\tx_{n-1}) \\
 & ~~
 + \dt
\sum_j a_{i,j} y'(\tx_{n-1})  + {\mathcal
  O}(\dt^2)\,, \quad i=1,2,\dots,
s \,,
\end{align*}
and therefore $\GLMU q_0=1$ and the abscissa vector $c=
\GLMA\mathbbm{1} + \GLMU q_1$, where $q_0=\{q\}_{i,0}$ and
$q_1=\{q\}_{i,1}$, $i=1,\dots,r$. We next 
combine \eqref{eq:prec:assump:in} and \eqref{eq:prec:assump:out}:
\begin{align*}
q_{i,0} (y(\tx_{n-1}) + \dt y'(\tx_{n-1}) ) + \dt q_{i,1} y'(\tx_{n-1}) & = q_{i,0} y(\tx_{n-1})
+ \dt q_{i,1} y'(\tx_{n-1}) \\
& \qquad + \dt \sum_j b_{1,j} y'(\tx_{n-1}) + {\mathcal
  O}(\dt^2)\,,
\end{align*}
where we have considered that $\GLMV=I$. The consistency
condition $\GLMB\mathbbm{1} = q_0$ follows.
%

%
\subsection{Order conditions for GEE methods\label{sec:Order:GEE}}
%
The order conditions are based on the algebraic representation of the
propagation of the B-series through the GL process as discussed in
Sec. \ref{sec:GL:order}. For form $\GLyy$, we need to set
$\xi_{\{1,2\}}(\emptyset)=1$ and $\xi_{\{1,2\}}(\Ttree_q)=0$, $q=1,2, 
\dots$, resulting in Runge-Kutta like conditions such as
$\widehat{\xi}_{\{i\}}(\BTreeIA)=\sum_j b_{i,j}$,
$[E\xi]_{\{i\}}(\BTreeIA)=1$, and $\widehat{\xi}_{\{i\}}(\BTreeIIA)=
\GLMB_{i,:} \GLMA \ones$,
$[E\xi]_{\{i\}}(\BTreeIIA)=1/2$, where $\widehat{\xi}$
represents the numerical output as introduced in 
\eqref{eq:GLM:error}, and $E\xi$ corresponds to the elementary weights
of the exact solution. 

Additional constraints are imposed so that Theorem
\ref{prop:two:methods} applies directly as a result of the GL
process. In particular, we need to enforce relations \eqref{eq:local:error:ode:two:methods}. To this end, we consider an order $p$ $\GLyy$ method by
setting
$E\xi_1(\Ttree)=\widehat{\xi}_1(\Ttree)=\widehat{\xi}_2(\Ttree)$, for
all $\Ttree \in \TreeSet_{p}$, and 
%
\begin{align}
\label{eq:GLMYY:Conditions:Error}
\gamma\left(E\xi_1(\Ttree) - \widehat{\xi}_1(\Ttree)
\right)=&E\xi_2(\Ttree) -
\widehat{\xi}_2(\Ttree) \,, \quad \Ttree \in \TreeSet_{p+1}\,, \gamma \ne 1\,,
\end{align}
assuming that the inputs of the GL process $\gamma{\xi}_1(\Ttree)={\xi}_2(\Ttree)$, $\Ttree \in
\TreeSet_{p+1}$. For instance, if $p=2$ and $r=2$, then
\eqref{eq:GLMYY:Conditions:Error}  yields:
$\gamma\left(\frac{1}{6}-\GLMB_{1,:} \GLMA \GLMA \ones\right)=
\left(\frac{1}{6}-\GLMB_{2,:} \GLMA \GLMA \ones\right)$ 
for
$\Ttree=\BTreeIVA$ and
$\gamma\left(\frac{1}{3}-\GLMB_{1,:} (\GLMA \ones)^2 \right)=
\left(\frac{1}{3}-\GLMB_{2,:} (\GLMA \ones)^2\right)$
for $\BTreeIIIA$, where the exponent is applied componentwise.
Then the error of the base method satisfies 
\begin{align}
  \nonumber
  \varepsilon_p=&\frac{\dt^p}{p!}\sum_{\Ttree \in \TreeSet_{p+1}}
(E\xi_1(\Ttree)-\widehat{\xi}_1(\Ttree)) F(\Ttree)(y) + {\mathcal
  O}(\dt^{p+1}) \,. 
\end{align}
%
Expression \eqref{eq:GLMYY:Conditions:Error} is equivalent to imposing
\eqref{eq:local:error:ode:two:methods}. We also impose
stability order \cite{Butcher_2002}
$\tilde{p}=p+3$: $\Phi(\exp(z),z) = {\mathcal O} (\dt^{\tilde{p}})$
defined in \eqref{eq:GL:StabFunction}, to
obtain robust methods.  

The two solutions that evolve through the GL process are connected
internally, and therefore the error estimation may be hindered in the
case of unstable dynamics as discussed in \cite{Prince_1978}. In
Fig. \ref{fig:nr:conv:principal:error} we illustrate such a
behavior, and in Fig. \ref{fig:cond:error:error} we show how the
convergence is affected if the decoupling conditions (see Proposition \ref{prop:GL:independence}) are satisfied in
turn. To this 
end, we require that the elementary differentials of the two methods
resulting from applying the GL method be independent from each other's entries for all
trees of order $p+1$ and $p+2$. According to \eqref{eq:GLM:error}, the
output weights depend on the method coefficients and the input
weights. This requirement therefore can be expressed as
\begin{align}
\label{eq:error:terms}
&\widehat{\xi}_{\{\ell\}}(\Ttree_j)(\xi_{\{1\}}(\Ttree_k),\xi_{\{2\}}(\Ttree_k),\xi_{\{1,2\}}(\Ttree_q))=\widehat{\xi}_{\{\ell\}}(\Ttree_j)(\xi_{\{\ell\}}(\Ttree_k),\xi_{\{1,2\}}(\Ttree_q))
\,,\\
\nonumber
& \qquad\forall j, k, ~~ \TreeOrder(\Ttree_k)\,,~\TreeOrder(\Ttree_j) \in
\{  p+1\,,~p+2\}\,, ~\TreeOrder(\Ttree_q) \in \{1,\dots,p\}\,,
\end{align}
where $\xi_{\{\ell\}}(\Ttree_j)$ is the coefficient of input $\ell$
corresponding to tree index $j$ and $\widehat{\xi}_{\{i\}}(\Ttree_k)$ is the
coefficient of GL output $i$ corresponding to tree index $k$. In other
words, output 1 that corresponds to tree index $j$ does not depend on
input 2 of tree index $k$, and the same for output 2 and input 1.

We now establish a mechanism by which conditions for the method's
coefficients are set so that \eqref{eq:error:terms} is satisfied.  
\begin{lemma}\label{lem:GL:stage:outpt}
The elementary weights of a GL method \eqref{eq:GLM:Global} with
$\GLMV=I$ satisfy 
\begin{align}
  \nonumber
\widehat{\xi}(\Ttree_{j})&= K + \xi(\Ttree_{j}) + \GLMB \GLMU
\xi(\Ttree_{[p-1]}) + G(\Ttree_{[\{1,2,\dots,p-2\}]})\,,
 ~ \TreeOrder(\Ttree_j) = p\,,~ j=1,2,\dots\,,
\end{align}
where $K$ is a constant that depends on the tree index,
$\Ttree_{[\ell]}$ are all the trees that correspond to order $\ell$
($\TreeOrder(\Ttree_i) = \ell$, $\forall i$),
 and $G$ is a function of $\Ttree$ of
order $1$ to $p-2$, and
$\xi(\Ttree_{q})=[\xi_1(\Ttree_{q}),\xi_2(\Ttree_{q})]^T$.  
\end{lemma}
\begin{proof}
For the first tree $\Ttree_\emptyset$ we have $\eta
D(\Ttree_\emptyset)=0$. The next tree is $\Ttree_1=\BTreeIA$, for which
$\eta D(\BTreeIA)=1$. Relation \eqref{eq:GLM:error} gives
\begin{align*}
\eta(\Ttree_1)&=\GLMA \eta D(\Ttree_1) + \GLMU \xi(\Ttree_1) = \GLMA
\ones_s + \GLMU \xi(\Ttree_1)\,.
\end{align*}
This is allowed by Lemma \ref{lem:ED:Independence}. Next, for order
2, we have $\eta
D(\Ttree_2)=\eta D(\BTreeIIA)=\eta(\Ttree_1)$ and 
\begin{align*}
\eta(\Ttree_2)&=\GLMA \eta D(\Ttree_2) + \GLMU \xi(\Ttree_2) = \GLMA
\cdot (\GLMA \ones+ \GLMU \xi(\Ttree_1) ) + \GLMU \xi(\Ttree_2)\,.
\end{align*}
For the next tree we have $\eta D(\Ttree_3)=\eta
D(\BTreeIIIA)=(\eta(\Ttree_1))^2$ and  
\begin{align*}
\eta(\Ttree_3)&=\GLMA (\GLMA \ones_s + \GLMU \xi(\Ttree_1))^2 +
\GLMU \xi(\Ttree_3) \,,
\end{align*}
where the power is taken componentwise. The last third-order tree
gives $\eta D(\Ttree_4)=\eta D(\BTreeIVA)=\eta(\Ttree_2)$ and 
\begin{align*}
\eta(\Ttree_4)&=\GLMA (\GLMA
\cdot (\GLMA \eta D(\Ttree_1) + \GLMU \xi(\Ttree_1) ) + \GLMU \xi(\Ttree_2)) +
\GLMU \xi(\Ttree_4)\\
&= \GLMA^3 \ones + \GLMA^2 \GLMU \xi(\Ttree_1) +
\GLMA \GLMU \xi(\Ttree_2)) + \GLMU \xi(\Ttree_4)\,.
\end{align*}
We then arrive at the following recurrence formula:
%
%
\begin{align}
  \label{eq:etaD:stages}
  \eta D(\Ttree_{[p]})&=W(\eta(\Ttree_{[\{1,2, \dots, p-2\}]})) +  \GLMU \xi(\Ttree_{[p-1]})\,,
\end{align}
where $W(\eta(\Ttree_{[\{1,2, \dots, p-2\}]}))$ is a function of
trees that are of order $1,2, \dots, p-2$ and of the method's
coefficients $\GLMA$ and $\GLMU$.
Similarly, one can verify that the recurrence for the output quantities satisfies
\begin{align}
\label{eq:etaD:output}
\widehat{\xi}(\Ttree_{[p]})&=\GLMB W(\eta(\Ttree_{[\{1,2, \dots, p-2\}}))  + \GLMB \GLMU \xi(\Ttree_{[p-1]})  +
\xi(\Ttree_{[p]})\,.
\end{align}
%
For $p=3$ we have trees with index 3 and 4, and the output is obtained again from
\eqref{eq:GLM:error} and using the above derivations as
\begin{align*}
\widehat{\xi}(\Ttree_3)&=\GLMB((\GLMA \ones_s)^2 + (\GLMU
\xi(\Ttree_1))^2)    + \xi(\Ttree_3) = \GLMB W(\xi(\Ttree_1)) + \xi(\Ttree_{3})\,,\\
\widehat{\xi}(\Ttree_4)&=\GLMB(\GLMA
\cdot (\GLMA \ones + \GLMU \xi(\Ttree_1) )) + \GLMB\GLMU
\xi(\Ttree_2) + \xi(\Ttree_4) = \GLMB W(\xi(\Ttree_1)) + \GLMB
\GLMU \xi(\Ttree_{2}) +  \xi(\Ttree_{4})\,, 
\end{align*}
which gives off the recurrence.
For arbitrary $p$ an inductive argument can be made. One needs to repeat the steps above by using
\eqref{eq:GLM:error} and computing the derivatives
\eqref{eq:etaD:stages}, which will yield $\GLMB \GLMU
\xi(\Ttree_{j})$ terms for $\TreeOrder(\Ttree_j) = p-1$,
$\xi(\Ttree_{j})$ for $\TreeOrder(\Ttree_j) = p$, and more complicated
expressions for $\TreeOrder(\Ttree_j) = 1,2,\dots,p-2$ as in
\eqref{eq:etaD:output}. 
\end{proof}

Proposition \ref{prop:GL:independence} provides sufficient conditions for the
independence assumptions \eqref{eq:error:terms}.

\begin{proposition}[Output independence of GEE
    method\label{prop:GL:independence}] 
A GEE method \eqref{eq:GLM:Global} with $r=2$ for which the off-diagonal elements of matrix $\GLMB\GLMU$ are
zero satisfies the
independence assumption \eqref{eq:error:terms}.
\end{proposition}
\begin{proof}
Using the results of Lemma \ref{lem:GL:stage:outpt}, we compute the
output $i$ for trees of order $p+1$ and assume that the input is
consistent of order $p$, that is, $\xi_i(\Ttree_{[\{1,2,\dots,p]\}})=0$. We obtain 
\begin{align*}
\widehat{\xi_i}(\Ttree_{[p+1]})&= K + \xi_i(\Ttree_{[p+1]}) + (\GLMB \GLMU)_{i,:}\,
\xi(\Ttree_{[p]}) + G(\Ttree_{[\{1,2,\dots,p-1\}]}) = K +
\xi_i(\Ttree_{[p+1]}) \,,~ i=1,2\,. 
\end{align*}
For $p+2$ and the fact that $\GLMB\GLMU$ is a
diagonal matrix, we obtain 
\begin{align*}
\widehat{\xi}_i(\Ttree_{[p+2]})&= K + \xi_i(\Ttree_{[p+2]}) + (\GLMB \GLMU)_{i,:}\,
\xi(\Ttree_{[p+1]}) + G(\Ttree_{[k \in \{1,2,\dots,p\}]}) \\
&=  K + \xi_i(\Ttree_{[p+2]}) + (\GLMB \GLMU)_{i,:}\,
\xi(\Ttree_{[p+1]}) \\
&=  K + \xi_i(\Ttree_{[p+2]}) + (\GLMB \GLMU)_{i,i}\,
\xi_i(\Ttree_{[p+1]})\,,~ i=1,2\,. 
\end{align*}
\end{proof}

A similar calculation for $p+3$ reveals that matrices
$\GLMB\GLMA\GLMU$ and $\GLMB \diag(\GLMA \ones) \GLMU$ need to have
only diagonal entries. These conditions are necessary when dealing
with mildly stiff systems and for long-time
simulations. Collectively, we call these decoupling conditions. We
note that such requirements were identified both by Murua and Makazaga
\cite{Murua_2000,Makazaga_2003} and Shampine
\cite{Shampine_1986b}. The work presented herein generalizes this
concept and applies it to a wider class of methods. 

\paragraph{Order barriers} The order barriers of an ($r$,$s$)-GL
method apply to the methods with built-in error estimates (GEE) of one
order lower, and this is rather an upper bound. For example, a GEE
method of order 2 has the same restrictions on $r$ and $s$ as does a GL
method of order 3. These restrictions are reasonable because otherwise the
underlying higher-order approximation discussed in Sec.
\ref{sec:Under:High:Order}  would violate the GL order 
barriers. 
\paragraph{Variable steps} The self-starting methods introduced in
this study are amenable to variable time steps. This fact results from the
discussion in Sec. \ref{sec:variable:time:step} and will be
illustrated numerically in 
Sec. \ref{sec:test:problems:experiments} (see
Fig. \ref{fig:nr:err:adapt}). However, the construction 
of GEE based on GL 
methods that propagate quantities from past steps (i.e., not
self-starting) with adaptive time steps will necessarily have
coefficients that depend on the time steps or by a rescaling strategy
\cite{Butcher_2003}, and will have to satisfy
the properties listed above. Moreover, a finalizing procedure might be
necessary at every step in order to extract the error
components. Constructing such a method can quite involved, but it is 
nonetheless possible, as illustrated by particular 
instances in \cite{Aid_1997}.  

%
\subsection{Optimal methods\label{sec:optimal:methods}}
%
We now discuss the need to balance the local truncation errors, which we
would like to be as small as possible, with the ability to capture the
global errors. Solving for the correction procedure is attractive
because it allows the reuse of methods with well-established
properties.  In particular, one may
consider methods that minimize the truncation errors. However, when
such optimal methods are used in the context of global error estimation, 
one must verify that the errors are still quantifiable. For
instance, if not all the truncation error terms are nonzero, then
special care needs to be exercised because some problems may render
the global error estimation ``blind'' to local error accumulation.

\begin{figure}
  \centering
\includegraphics[width=.4\textwidth]{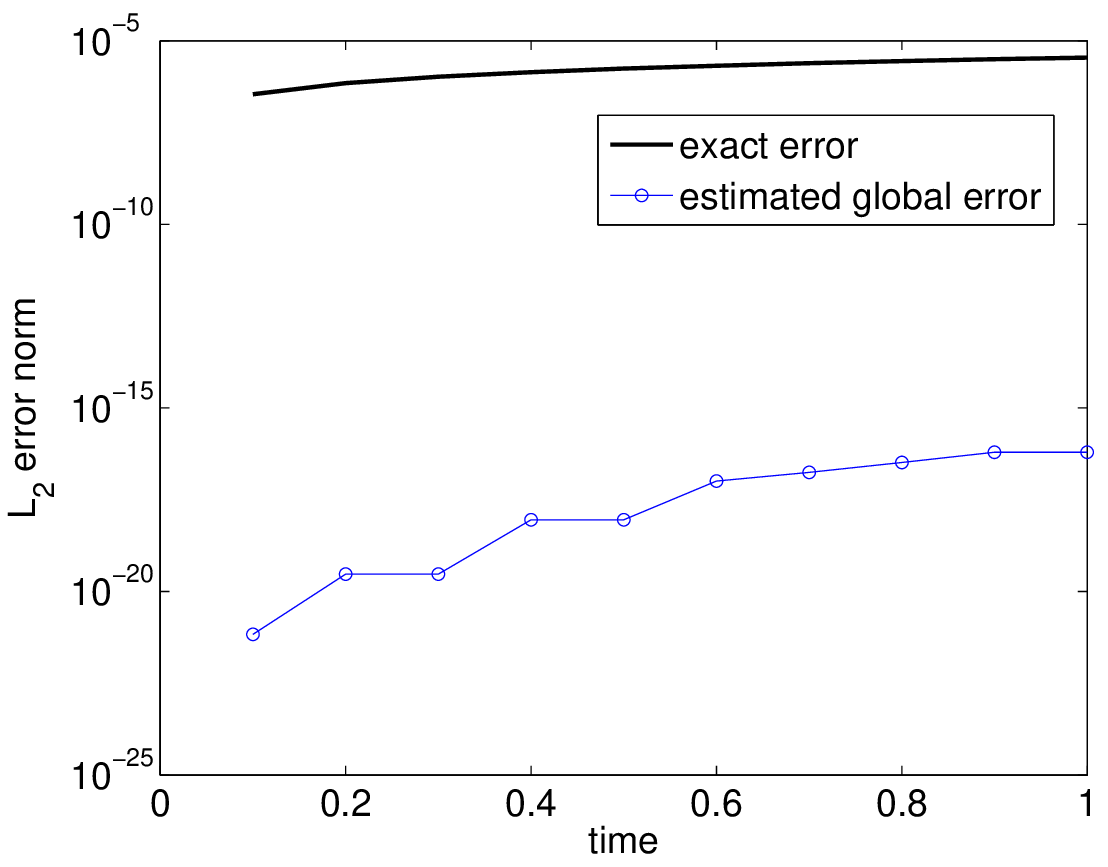}
  \caption{Failure to capture the
  global errors correctly for system $y_1'=1$, $y_2'=\kappa_2 y_1^3$,
$y_3'=\kappa_3 y_1^4$ solved with 
RK3(2)G1 \eqref{eq:RK3Q:in}
 \cite{Dormand_1984}.\label{fig:DormandRK3Q}}
\end{figure}

To illustrate this rather subtle point, we consider using the
procedure for solving for the correction \eqref{eq:Solving:Correction:ODE} 
with method RK3(2)G1 \eqref{eq:RK3Q:in} as introduced in
\cite{Dormand_1984}. This is a third-order scheme; however, it
 has no errors that correspond to fourth-order trees  $\BTreeVIA$
 and $\BTreeVIIIA$ but does not resolve
exactly $\BTreeVIIA$ and $\BTreeVA$; otherwise it would have been a
fourth-order method. With the aid of Lemma \ref{lem:ED:Independence} we
construct a simple problem: $y_1'=1$, $y_2'=\kappa_2 y_1^3$,
$y_3'=\kappa_3 y_1^4$, where $\kappa_i$ are some constants (e.g., $\kappa_2=1/6$ and $\kappa_3=4$). For this problem, the RK3(2)G1
is an order 4 method because the tall tree that would have affected the third
component is matched exactly by this method. This means that the base
method $y$ has the same order as the higher-order companion,
$\widehat{y}$. Therefore, the third component
can cause the results to be unreliable. In Fig. \ref{fig:DormandRK3Q}
we show the third component, which confirms the inadequacy in the error
estimation procedure. We note that this analysis also applies to the
schemes proposed in \cite{Murua_2000,Makazaga_2003}.

%
\subsection{Second-order explicit Runge-Kutta-like methods}
%
We now introduce a few methods of type [A:GLMGEE] \eqref{eq:GE:GL}. We
begin with a detailed inspection of second-order methods.
Schemes with $s=2$ are not possible because that would imply that one
can have an explicit third-order method by \eqref{eq:high:order:estimate} with only two
stages, which is a statement that is easy to disprove.  

A method with $s=3$ and $\gamma=0$ in $\GLye$ form is given by the
following tableaux,  
\begin{align}
\label{GLM:E:s2:p2:RK:1}
\GLMMye=\left[
\begin{array}{ccc|cc}
0     & 0     & 0   & 1  &   0 \\ 
1     & 0     & 0   & 1  & 10 \\
1/4   & 1/4   & 0   & 1  &  -1 \\
\hline
1/12  & 1/12  & 5/6 & 1  &   0\\
1/12  & 1/12  &-1/6 & 0  &   1\\
\end{array}
\right]\,,
\end{align}
where the four blocks represent
$(\GLMA_{y\varepsilon},\GLMU_{y\varepsilon},\GLMB_{y\varepsilon},\GLMV_{y\varepsilon})$
as discussed above. The diagonal $\GLMB\GLMU$ condition is satisfied
($[\GLMB\GLMU]$); however, the $\GLMB\GLMA\GLMU$ one is not
($[\cancel{\GLMB\GLMA\GLMU}])$.

Method \eqref{GLM:E:s2:p2:RK:1} can then be expressed as follows:
\begin{subequations}
\label{GLM:E:s2:p2:RK:1:det}
\begin{align}
Y_1 =& y^{[n-1]} \,,\\
Y_2 =& y^{[n-1]} + 10 \varepsilon^{[n-1]} + \dt f(Y_1) \,,\\
Y_3 =& y^{[n-1]} - \varepsilon^{[n-1]} + \dt \left(\frac{1}{4} f(Y_1) +
\frac{1}{4} f(Y_2) \right) \,, \\
y^{[n]}=&y^{[n-1]} +   \dt \left(\frac{1}{12} f(Y_1) +
\frac{1}{12} f(Y_2)  + \frac{5}{6} f(Y_3)\right) \,, \\
\varepsilon^{[n]}=&\varepsilon^{[n-1]} +   \dt \left(\frac{1}{12} f(Y_1) +
\frac{1}{12} f(Y_2)  - \frac{1}{6} f(Y_3)\right)\,.
\end{align}
\end{subequations}
In \eqref{GLM:E:s2:p2:RK:1:det} we note the Runge-Kutta structure;
however, we see that the defect takes an active role in the stage calculations. Using \eqref{eq:GLyy:to:GLye}, we obtain the $\GLyy$ form as
\begin{align}
\label{GLM:E:s2:p2:RK:1:GLyy}
\GLMMyy=\left[
\begin{array}{ccc|cc}
0     & 0     & 0   & 1  &   0 \\ 
1     & 0     & 0   & -9 &  10 \\
1/4   & 1/4   & 0   & 2  &  -1 \\
\hline
1/12  & 1/12  & 5/6 & 1  &   0\\
1/6   & 1/6   & 2/3 & 0  &   1\\
\end{array}
\right]\,.
\end{align}
In particular, \eqref{GLM:E:s2:p2:RK:1:GLyy} is expressed as
\begin{subequations}
\label{GLM:E:s2:p2:RK:1:det:GLyy}
\begin{align}
Y_1 =& y^{[n-1]} \,,\\
Y_2 =& -9 y^{[n-1]} + 10 \tilde{y}^{[n-1]} + \dt f(Y_1) \,,\\
Y_3 =& 2 y^{[n-1]} - \tilde{y}^{[n-1]} + \dt \left(\frac{1}{4} f(Y_1) +
\frac{1}{4} f(Y_2) \right) \,, \\
\label{GLM:E:s2:p2:RK:1:det:GLyy:yn}
y^{[n]}=&y^{[n-1]} +   \dt \left(\frac{1}{12} f(Y_1) +
\frac{1}{12} f(Y_2)  + \frac{5}{6} f(Y_3)\right) \,,\\
\tilde{y}^{[n]}=&\tilde{y}^{[n-1]} +   \dt \left(\frac{1}{6} f(Y_1) +
\frac{1}{6} f(Y_2)  + \frac{2}{3} f(Y_3)\right)\,,\quad
\varepsilon^{[n]} = \tilde{y}^{[n]} -  y^{[n]}\,.
\end{align}
\end{subequations}
Here we note the explicit contribution of two solutions. A solution of
order 3 is obtained according to \eqref{eq:high:order:estimate} by
$\widehat{y}^{[n]}=\tilde{y}^{[n]}$ because $\gamma=0$. Moreover, a local error estimate
for $y^{[n]}$ in \eqref{GLM:E:s2:p2:RK:1:det:GLyy:yn} corresponds to 
\begin{align}
\label{eq:GLM:E:s2:p2:RK:1:loc}
\varepsilon_{\rm loc}=\varepsilon^{[n]}-\varepsilon^{[n-1]}= \dt \left(\frac{1}{12} f(Y_1) +
\frac{1}{12} f(Y_2)  - \frac{1}{6} f(Y_3)\right)\,,
\end{align}
which is an obvious statement. This is also obtained by replacing
$\tilde{y}^{[n-1]}$ by  $y^{[n]}$ in the right-hand sides of
\eqref{GLM:E:s2:p2:RK:1:det:GLyy}
and taking the differences between the two solutions or setting
$\varepsilon^{[n-1]}=0$ in \eqref{GLM:E:s2:p2:RK:1:det}. Additional second-order methods are given in Appendix 
\ref{app:more:second:orders}. We remark that methods using [A:SolCor] require
at least four stages. While this seems like a marginal improvement, it it
expected to reap more benefits when considering higher-order methods. 

%
\subsection{Third-order explicit Runge-Kutta-like methods}
%

Closed-form solutions were difficult to obtain for methods of order
3. We therefore explored the space of such methods using a numerical
optimization such as in \cite{Constantinescu_A2010b}. One method of
order 3 with $\gamma=0$, $s=5$ stages, $[\GLMB\GLMU,\GLMB\GLMA\GLMU]$, and having 
significant negative real axis stability was found to have the
following coefficients up to 40 digits accuracy:
\begin{align}
\label{eq:GLyy:third:order}
\begin{array}{lll}
\ST
a_{2, 1} = -\frac{2169604947363702313}{24313474998937147335}, &
a_{3, 1} = \frac{46526746497697123895}{94116917485856474137}, &
a_{3, 2} = -\frac{10297879244026594958}{49199457603717988219}, \\
\ST
a_{4, 1} = \frac{23364788935845982499}{87425311444725389446}, &
a_{4, 2} = -\frac{79205144337496116638}{148994349441340815519},& 
a_{4, 3} = \frac{40051189859317443782}{36487615018004984309}, \\
\ST
a_{5, 1} = \frac{42089522664062539205}{124911313006412840286}, &
a_{5, 2} = -\frac{15074384760342762939}{137927286865289746282}, &
a_{5, 3} = -\frac{62274678522253371016}{125918573676298591413},\\ 
\ST
a_{5, 4} = \frac{13755475729852471739}{79257927066651693390}, &
b_{1, 1} = \frac{61546696837458703723}{56982519523786160813},& 
b_{1, 2} = -\frac{55810892792806293355}{206957624151308356511},\\ 
\ST
b_{1, 3} = \frac{24061048952676379087}{158739347956038723465}, &
b_{1, 4} = \frac{3577972206874351339}{7599733370677197135}, &
b_{1, 5} = -\frac{59449832954780563947}{137360038685338563670},\\ 
\ST
b_{2, 1} = -\frac{9738262186984159168}{99299082461487742983}, &
b_{2, 2} = -\frac{32797097931948613195}{61521565616362163366}, &
b_{2, 3} = \frac{42895514606418420631}{71714201188501437336}, \\
\ST
b_{2, 4} = \frac{22608567633166065068}{55371917805607957003}, &
b_{2, 5} = \frac{94655809487476459565}{151517167160302729021}, &
u_{1, 1} = \frac{70820309139834661559}{80863923579509469826}, \\
\ST
u_{1, 2} = \frac{10043614439674808267}{80863923579509469826}, &
u_{2, 1} = \frac{161694774978034105510}{106187653640211060371},& 
u_{2, 2} = -\frac{55507121337823045139}{106187653640211060371}, \\
\ST
u_{3, 1} = \frac{78486094644566264568}{88171030896733822981}, &
u_{3, 2} = \frac{9684936252167558413}{88171030896733822981}, &
u_{4, 1} = \frac{65394922146334854435}{84570853840405479554}, \\
\ST
u_{4, 2} = \frac{19175931694070625119}{84570853840405479554}, &
u_{5, 1} = \frac{8607282770183754108}{108658046436496925911}, &
u_{5, 2} = \frac{100050763666313171803}{108658046436496925911}.
\end{array}
\end{align}
We note that this is not an optimal method. It is just an example that
was relatively easy to obtain and will be used in the numerical
experiments. We note that method RK3(2)G1 \eqref{eq:RK3Q:in} as introduced in
\cite{Dormand_1984} requires 8 stages (that number reduces to 7 because of
its FSAL property).

%
\section{Relationships with other global error estimation strategies\label{sec:relations}}
%
Here we discuss the relationship between our approach and the existing
strategies that we focus on in this study. We show how the latter are
particular instantiations of the strategy introduced here. This
inclusion is facilitated by the use of Lemma \ref{lem:GLMyy:GLMye},
which reveals a linear relationship between propagating two solutions
and propagating one solution and its defect. We discuss below in some detail
the procedure for solving for the correction and the extrapolation
approach. Method [A:ZaPr] follows from the same discussions as those
for [A:SolCor]. To express algorithm [A:SoErEq] as a GL method,
we need more involved manipulations of these methods. Method 
[A:ExPrErEq] with the exact principal error equation
\eqref{eq:ExactPrincipalError} can 
obviously be represented as a GL scheme. Methods that implicitly solve the
error equation can also be represented as GL schemes; however, in this
study we  will not expand on this point. 

%
\subsection{Approach for solving for the correction\label{sec:Zadunaisky}}
%
Let us consider the Runge-Kutta methods that integrate the global
errors introduced by \cite{Dormand_1984,Dormand_1985,Dormand_1994,Murua_2000,Makazaga_2003}. The RK tableau
is defined by the triplet ($\RKA$, $\RKb$, $\RKc$) and the
interpolation operators by ($\DRKB$, $\DRKD$), where  $\DRKB \cdot
[\theta^0, \theta^1, \dots, \theta^s]^T$  yields the interpolant weight
vector and  $\DRKD \cdot [\theta^0, \theta^1, \dots, \theta^s]^T$ yields
its derivative. In particular, $\DRKD_{ij}=\DRKB_{ij} \cdot j$,
$j=1,\dots,s$. Denote by $b^*_i(\theta)= \sum_{j=1}^{p^*} \DRKB_{ij}
\theta^j$, $d^*_i(\theta)= \sum_{j=1}^{p^*} \DRKD_{ij} \theta^j$, and
consider the dense output formula given by 
\begin{align}
\nonumber
P(\tx+\theta \dt) = y_n + \theta \dt \sum_{i=1}^s b^*_i f_i\, ~
\textnormal{and} ~
P'(\tx+\theta\dt) = \theta \dt \sum_{i=1}^s d^*_i f_i\,
\end{align}
and the error equation that is being solved is
\eqref{eq:Solving:Correction:ODE:z}
($ \varepsilon'(\tx)  = f(\tx,P(\tx)+\varepsilon(\tx)) - P'(\tx)$).
%
We denote by $\ODRKB =\diag\{\RKc\} \cdot \DRKB \cdot W(\RKc)^T$, where
$W(\RKc)$ is the Vandermonde matrix with entries $\RKc$; that is,
$\{W(\RKc)\}_{ij}=\RKc_i^{j-1}$; and $\ODRKD = \DRKD \cdot W(\RKc)^T$. 
The resulting method cast in GL format \eqref{eq:GLM:Global} is 
\begin{align}
\label{eq:Solving:Correction:GLM}
\left[
\begin{array}{c}
Y_1\\
Y_2\\
y_{n+1}\\
\varepsilon_{n+1}
\end{array} 
\right] 
=
\left[
\begin{array}{cc|cc}
\RKA & 0 &\ones_s  & 0\\
\ODRKB - \RKA \ODRKD &\RKA  &\ones_s & \ones_s\\
\hline
\RKb^T &0 &1 &0\\
-\RKb^T \ODRKD &  \RKb^T  &0 &1
\end{array} 
\right] 
\left[
\begin{array}{c}
\dt f(Y_1)\\
\dt f(Y_2)\\
y_{n}\\
\varepsilon_{n}
\end{array} 
\right] \,.
\end{align}
Here we express the method for a scalar problem, in order to avoid the tensor
products, and we represent the stacked stages in $Y_{\{1,2\}}$. 
For example, method RK3(2)G1 \cite{Dormand_1984} is given by the
following Butcher tableau: 
\begin{align}
\label{eq:RK3Q:in}
& \begin{array}{c|cccc}
\ST 0          &         0\\
\ST \frac{1}{2}& \frac{1}{2}   &   0 \\
\ST 1          &   -1          &   2        &0\\
\ST 1          &   \frac{1}{6} &\frac{2}{3} &\frac{1}{6}&0\\
\cline{1-5}
\ST &  \frac{1}{6} &\frac{2}{3} &\frac{1}{6}&0
\end{array}\,,
\quad
\begin{array}{l}
\DRKB=
\left[
\begin{array}{ccccc}
1&-\frac{3}{2}&\frac{2}{3}\\
0&2&-\frac{4}{3}\\
0& \frac{3}{6}&-\frac{2}{6}\\
0&-1&1
\end{array}
\right]\\
\ST
\DRKD_{ij}=\DRKB_{ij} \cdot j
\end{array}
\,.
\end{align}
The equations to be solved when using the procedure for solving for
the correction [A:SolCor] are
then \eqref{eq:Solving:Correction:ODE}; however, one can show that
they are equivalent to solving \eqref{eq:Solving:Correction:GLM} and
using the strategy [A:GLMGEE] \eqref{eq:GLM:Global} introduced
here. The explicit coefficients are listed in tableau
\eqref{eq:RK3Q:in:GLM}. This strategy of estimating the global errors
is called Runge-Kutta triplets; a similar discussion can be extended
for  the peer-triplets strategy \cite{Weiner_2014}.
The Zadunaisky procedure can be shown to have a similar
interpretation; however, it is a little more expensive, and the
analysis has to be carried over several steps. We will draw
conclusions about its behavior by using [A:SolCor] as a proxy.

%
\subsection{Global error extrapolation}
%
Let us consider again the Runge-Kutta methods defined by the triplet
($\RKA$, $\RKb$, $\RKc$) of order $p$. By applying
\eqref{eq:glabal:error:extrapolation} we obtain the method in the GL 
format,   
\begin{align}
\label{eq:Extrapolation:GLM}
\left[
\begin{array}{c}
Y_1\\
Y_2\\\
Y_3\\
y_{n+1}\\
\varepsilon_{n+1}
\end{array} 
\right] 
=
\left[
\begin{array}{ccc|cc}
\RKA & 0 &0  &\ones_s  & 0\\
0 & \frac{1}{2} \RKA &0  &\ones_s  &\beta^{-1} \ones_s\\
0 & \frac{1}{2} \RKb^T \otimes \ones_s & \frac{1}{2} \RKA  &\ones_s  &
\beta^{-1} \ones_s\\ 
\hline
\RKb^T &0 & 0&1 &0\\
-\beta \RKb^T &  \frac{\beta}{2} \RKb^T&  \frac{\beta}{2} \RKb^T  &0 &1
\end{array} 
\right] 
\left[
\begin{array}{c}
\dt f(Y_1)\\
\dt f(Y_2) \\
\dt f(Y_3)\\
y_{n}\\
\varepsilon_{n}
\end{array} 
\right] \,,
\end{align}
where $\beta=\frac{1}{1-\gamma}$, $\gamma=1/2^p$, and $Y_{\{1,2,3\}}$ are
the $s$-stage vectors corresponding to the original method stacked on
top of each other. This is a method of type 
\eqref{eq:GLM:Global}.

%
\section{Numerical results\label{sec:Numerical:Exeriments}}
%
%
In this section we present numerical results with a detailed set
of test problems. 
%
%
\subsection{Test problems\label{sec:test:problems}}
%
We consider a set of simple but comprehensive test problems.
%
%
%

Problem [Prince42] is defined in \cite{Prince_1978} (4.2) by
\begin{subequations}
\label{eq:Prince_1978_(4.2)}
\begin{align}
\label{eq:Prince_1978_(4.2):eq}
y'&= y-\sin(\tx)+\cos(\tx)\,,~~ y(0)=\kappa \\
\label{eq:Prince_1978_(4.2):sol}
y(\tx)&= \kappa*\exp(\tx)+\sin(\tx)\,.
\end{align}
\end{subequations}
Here we take $\kappa=0$. As a direct consequence of
\eqref{eq:Prince_1978_(4.2):sol}, we see that that any perturbation of
the solution $y$,  such as numerical errors, leads to exponential
growth. Therefore we have an unstable dynamical system; and even if we
start with $\kappa=0$, numerical errors will lead to an exponential
solution growth. This is a classical example that is used to show
the failure of local error estimation in general and of global error
estimation by using Algorithm [A:ExPrErEq] \eqref{eq:ExactPrincipalError}
\cite{Prince_1978} in particular.
%
%
%
%

A similar problem [Kulikov2013I] is defined by Kulikov \cite{Kulikov_2013} by
%
\begin{align}
\label{eq:Kulikov_2013_I}
y_1'=2 \tx \, y_2^{1/5} \, y_4\,,~
y_2'=10 \tx \exp(5(y_3-1)) \, y_4 \,,~
y_3'=2 \tx \, y_4 \,,~
y_4'=-2 \tx \ln(y_1)\,,
\end{align}
so that $y_1(\tx)=\exp(\sin(\tx^2))$, $y_2(\tx) = \exp(5\sin(\tx^2))$,
$y_3(\tx)=\sin(\tx^2)+1$, $y_4(\tx) =\cos(\tx^2)$.
%
%
This problem is nonautonomous and exhibits unstable modes later.
%
%
%
%
%

To illustrate the error behavior over long time integration, we
consider problem [Hull1972B4], which is a nonlinear ODE defined in \cite{Hull_1972} (B4) by
%
\begin{align}
\label{eq:Hull1972_B4}
y_1'=-y_2-\frac{y_1 y_3}{\sqrt{y_1^2+y_2^2}} \,,~
y_2'=y_1-\frac{y_2 y_3}{\sqrt{y_1^2+y_2^2}} \,, ~
y_3'=\frac{y_1}{\sqrt{y_1^2+y_2^2}}  \,,
\end{align}
with $y_0=[3,0,0]^T$.
%

The last problem [LStab2] is used to assess linear stability
properties of the proposed numerical methods.
%
%
%
\begin{subequations}
\label{eq:lstab2}
\begin{align}
\label{eq:lstab2:eq}
&y'=A y\,,~ y(0)=[y_1(0),y_2(0)]^T\,,~A=\left[ \begin{array}{cc} a & -b \\ b & a  \end{array} \right] \,,~ \Lambda(A) = \{a+ib, a-ib\}\,,\\
\label{eq:lstab2:sol}
&\left\{ \begin{array}{l} y_1(\tx) = \exp(at) \left(y_2(0) \cos(b\tx) - y_1(0) \sin(b\tx) \right)\\
y_2(\tx) = \exp(at) \left(y_1(0) \cos(b\tx) + y_2(0) \sin(b\tx) \right)
\end{array} \right.\,.
\end{align}
\end{subequations}
This problem allows one to choose the position of the eigenvalues of the Jacobian,
$\Lambda(A)$, in order to simulate problems with different spectral properties.

%
\subsection{Numerical experiments\label{sec:test:problems:experiments}}
%

We begin with simple numerical experiments that show when local error
control without global error estimates is not suitable. We therefore
compare the result of well-tuned numerical 
integrators that use local error control with the global error
estimates for the same problem. We use Matlab's
{\rm ode45} integrator with different tolerances whenever we refer to
methods with local error estimation. 

In Fig. \ref{fig:unstable} we show the errors over time for problems
[Prince42] \eqref{eq:Prince_1978_(4.2)} and [Kulikov2013I]
\eqref{eq:Kulikov_2013_I}. These problems are solved by using LEE methods,
Figs. \ref{fig:unstable:a}-\ref{fig:unstable:b}, and GEE methods, \eqref{GLM:E:s2:p2:RK:A2} and
\eqref{eq:GLMyy:A9}, Figs. 
\ref{fig:unstable:c}-\ref{fig:unstable:d}; furthermore, the other GEE
methods introduced above give similar results. The
absolute error tolerance for LEE control is set to 1e-02. The
methods with LEE systematically underestimate the global error levels as
expected, whereas the methods with GEE capture the errors
relatively well. Moreover, the global errors are captured well across
components, as shown in Fig. \ref{fig:unstable:d}. The LEE results
are intended to show unreliability in accuracy; however, this setting
is not suitable for comparing efficiency between LEE and GEE because
LEE methods are not designed to account for or estimate the global
error. The point of GEE methods is that they provide global error
estimates whereas LEE methods do not; however, in this case the GEE
methods do not perform step-size adaptivity to improve efficiency, as
is done with the LEE methods.

%
%

\begin{figure}
\begin{center}
\subfigure[{[Prince42] \eqref{eq:Prince_1978_(4.2)} with LEE}]{\includegraphics[width=0.45\textwidth]{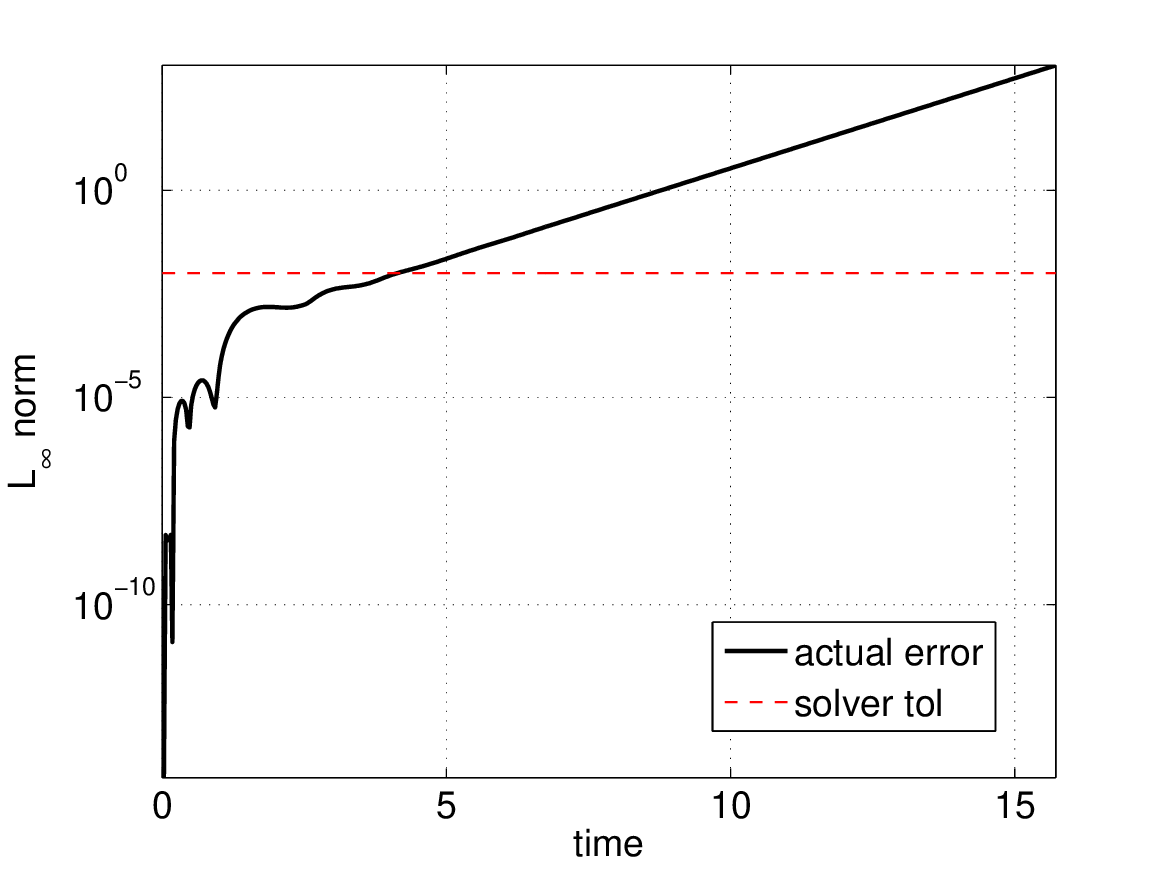}\label{fig:unstable:a}}
\subfigure[{[Kulikov2013I] \eqref{eq:Kulikov_2013_I} with LEE}]{\includegraphics[width=0.45\textwidth]{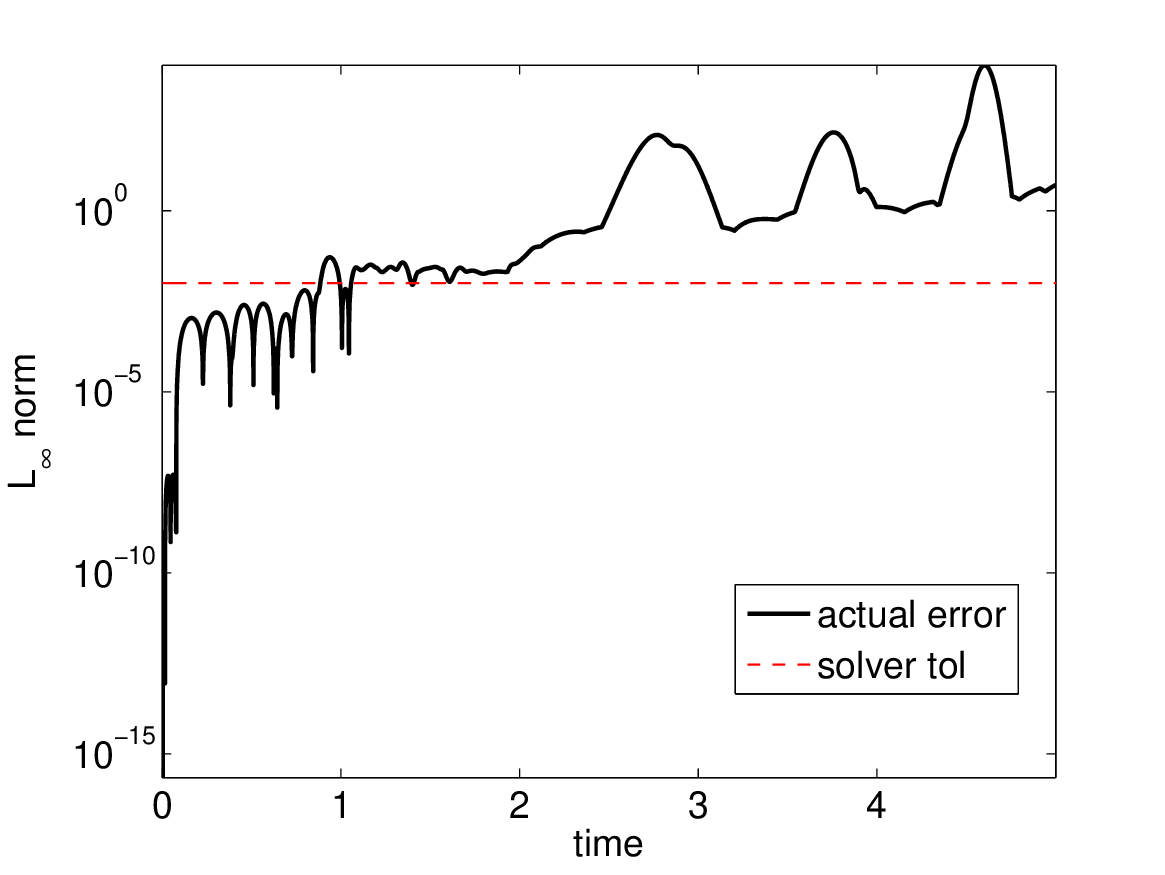}\label{fig:unstable:b}}
\subfigure[Problem (a) with GEE]{\includegraphics[width=0.45\textwidth]{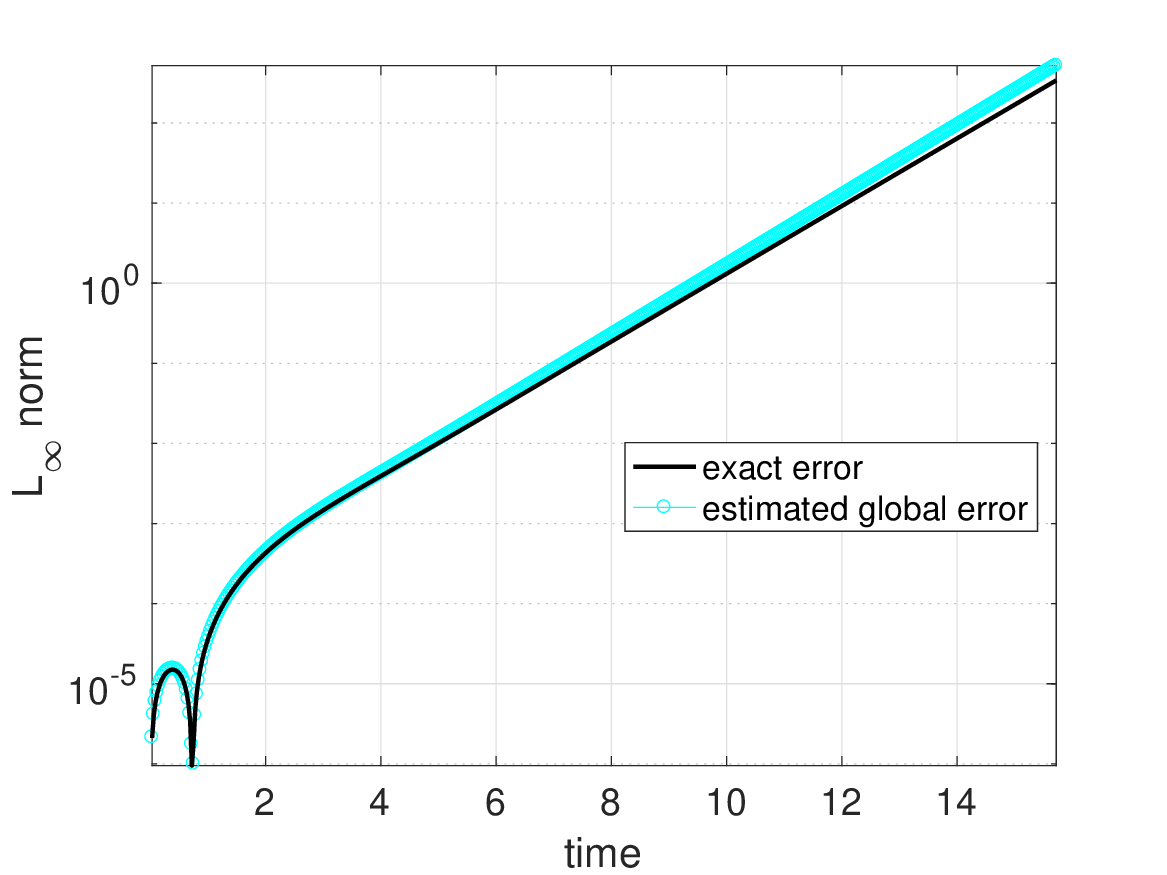}\label{fig:unstable:c}}
\subfigure[Problem (b) with GEE, by components]{\includegraphics[width=0.45\textwidth]{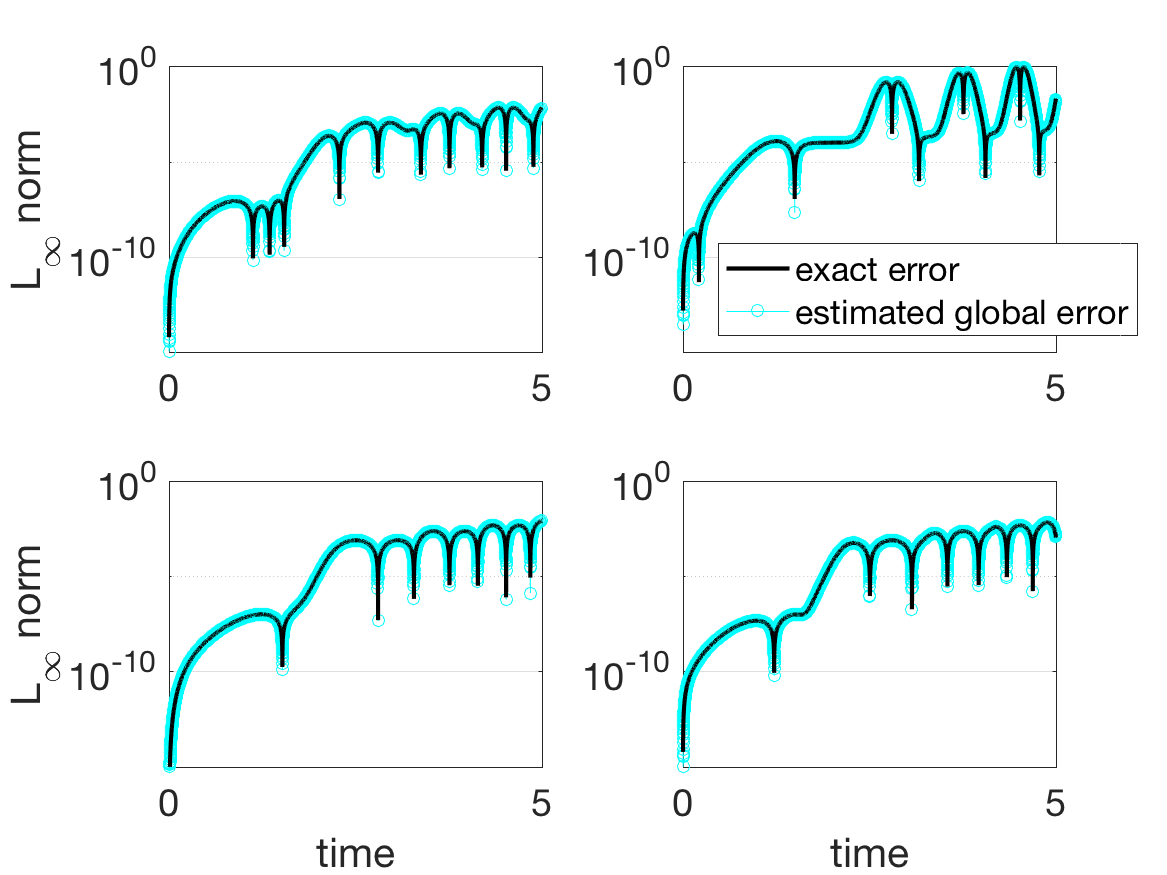}\label{fig:unstable:d}}
\end{center}
\caption{Errors when solving problems with unstable modes by using local
  estimated error (LEE) control and global error estimation (GEE).  The
  absolute error tolerance for LEE control is set to 1e-02. The GEE
  method used here is \eqref{GLM:E:s2:p2:RK:A2} with $\dt=3e-02$ in
  (c) and \eqref{eq:GLMyy:A9} with $\dt=5e-04$ in (d). GEE captures the
  errors well, whereas LEE underestimates them severely.\label{fig:unstable}} 
\end{figure}
%
%
%
%

In Fig. \ref{fig:long:term} we show the error behavior for problem
[Hull1972B4] \eqref{eq:Hull1972_B4} when long integration windows are
considered. For LEE we set the absolute tolerance to 1e-05. In this case
we observe an error drift to levels of 1e-03 over 1,000 time
units. The method with GEE
\eqref{eq:GLMyy:A9} can follow closely the error in
time. We note that methods \eqref{GLM:E:s2:p2:RK:1} and 
  \eqref{GLM:E:s2:p2:RK:A2} with $[\cancel{\GLMB\GLMA\GLMU}]$ do not
  perform well. Therefore, the built-in error estimator in the GEE method to
maintain its accuracy over time, we need to decouple the two embedded
schemes further by ensuring that $\GLMB\GLMA\GLMU$ is a diagonal
matrix; moreover, both GEE methods that satisfy this property
perform well on this test case. This aspect is observed for all
methods introduced here, discussed further below, and illustrated in
Fig. \ref{fig:cond:error:error}. 

\begin{figure}
\begin{center}
\subfigure[LEE for \eqref{eq:Hull1972_B4}]{\includegraphics[width=0.32\textwidth]{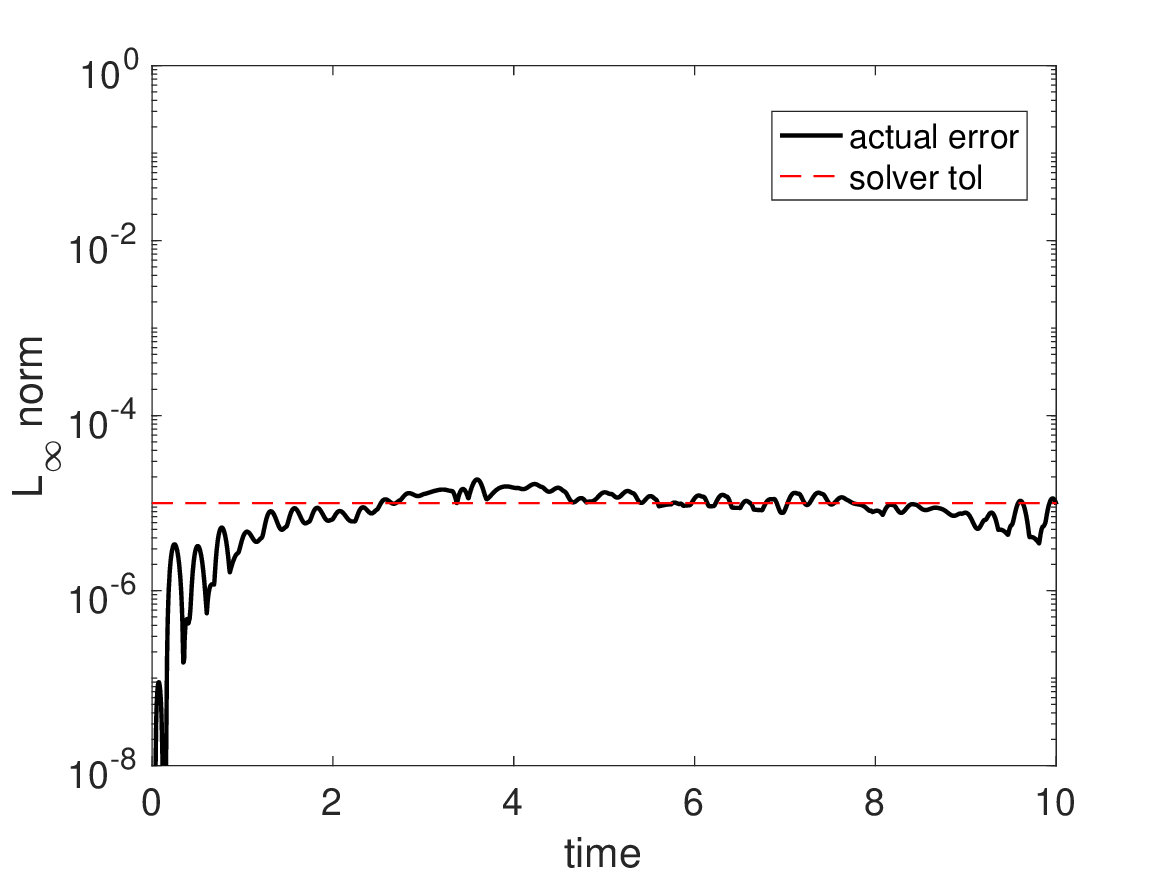}\label{fig:long:term:a}}
\subfigure[LEE for long times]{\includegraphics[width=0.32\textwidth]{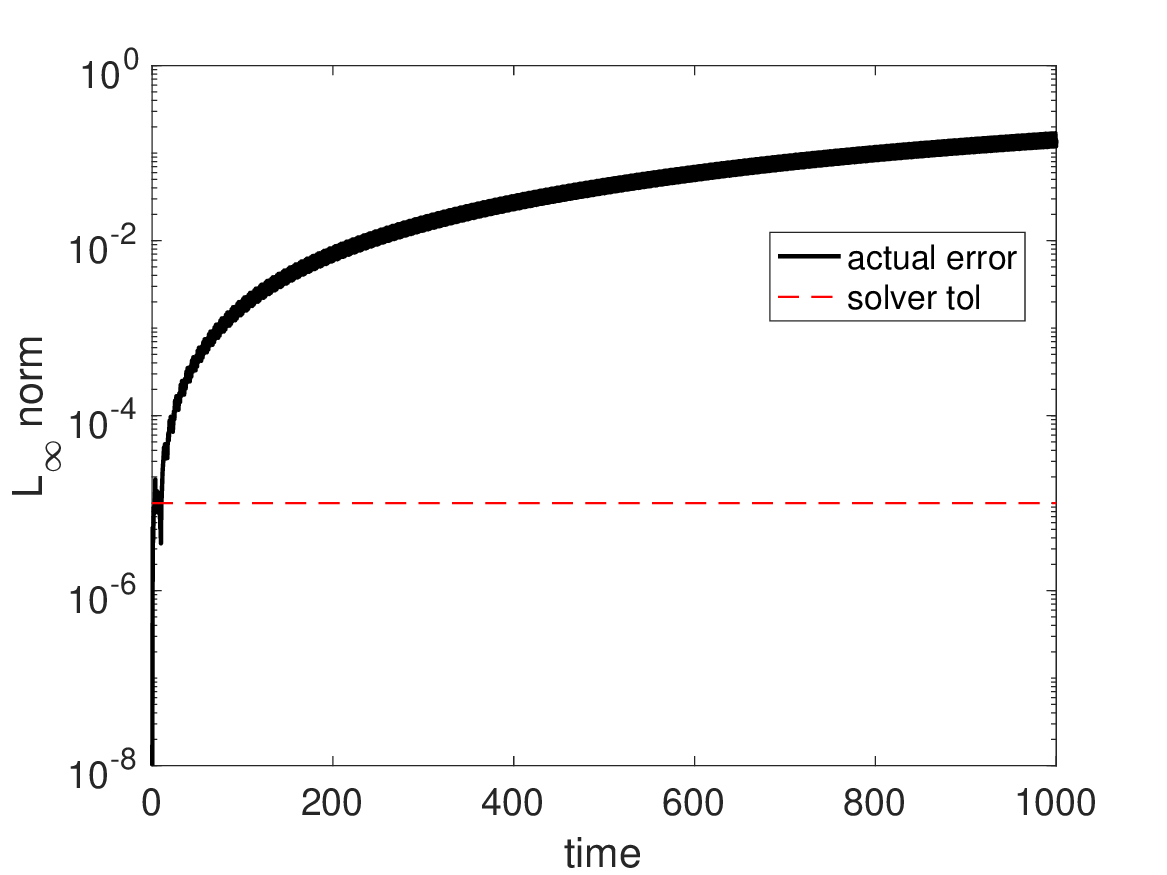}\label{fig:long:term:b}}
\subfigure[GEE for long times]{\includegraphics[width=0.32\textwidth]{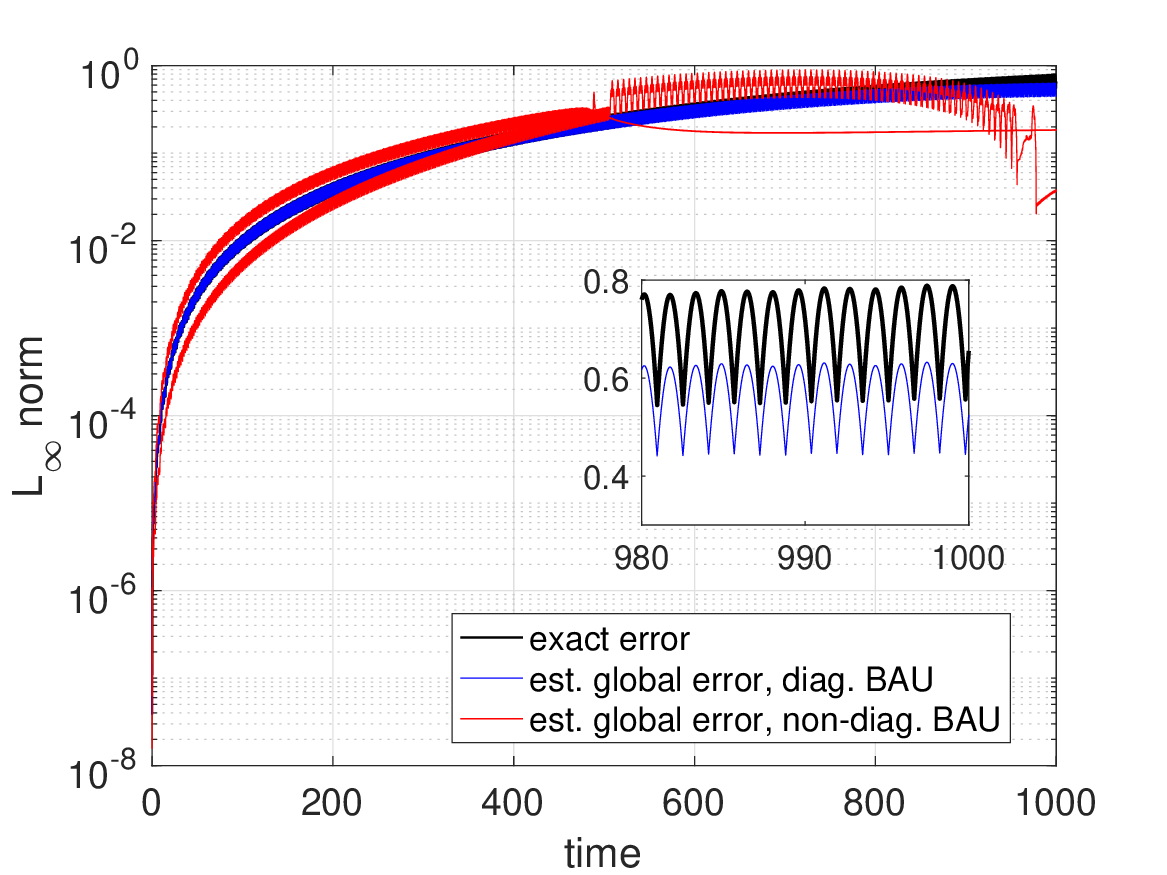}\label{fig:long:term:c}}
\end{center}
\caption{Errors when solving [Hull1972B4] \eqref{eq:Hull1972_B4} with
  LEE and GEE. For LEE we set the absolute tolerance to 1e-05.  {\rm (a)} During
  short integration times LEE satisfies the error tolerance
  well.  {\rm (b)} For longer times, however, we see an expected drift to error
  levels of 1e-03.  {\rm (c)} GEE method with $\dt=0.005$ and diagonal $\GLMB\GLMA\GLMU$ \eqref{eq:GLMyy:A9} gives accurate error estimates even over
  long times, while GEE with non-diagonal $\GLMB\GLMA\GLMU$
  \eqref{GLM:E:s2:p2:RK:1} and 
  \eqref{GLM:E:s2:p2:RK:A2} do not.\label{fig:long:term}} 
\end{figure}
%

%
%

\paragraph{Convergence analysis} We next analyze the convergence
properties of the methods discussed 
here. In Fig. \ref{fig:nr:conv} we show the convergence of the
solution and of the error estimate. Here we illustrate the convergence
of GEE methods of orders 2  \eqref{GLM:E:s2:p2:RK:A4} and 3
\eqref{eq:GLyy:third:order}  for problem [Prince42]
\eqref{eq:Prince_1978_(4.2)}. All GEE methods introduced in this study
converge with their 
prescribed order; moreover, the error estimate also converges with
order $p+1$, as expected from \eqref{eq:high:order:estimate}.  
In Fig. \ref{fig:nr:conv:principal:error} we show the behavior of
global error estimation when  using methods satisfying the exact principal error equation. Here we use method \cite[(3.11)]{Prince_1978}, which fails to
capture the error magnitude as discussed in \cite{Prince_1978} because
the estimated error is several orders of magnitude smaller than the
true global error. 
\begin{figure}
\begin{center}
\subfigure[Second-order method \eqref{GLM:E:s2:p2:RK:A4}]{\includegraphics[width=0.45\textwidth]{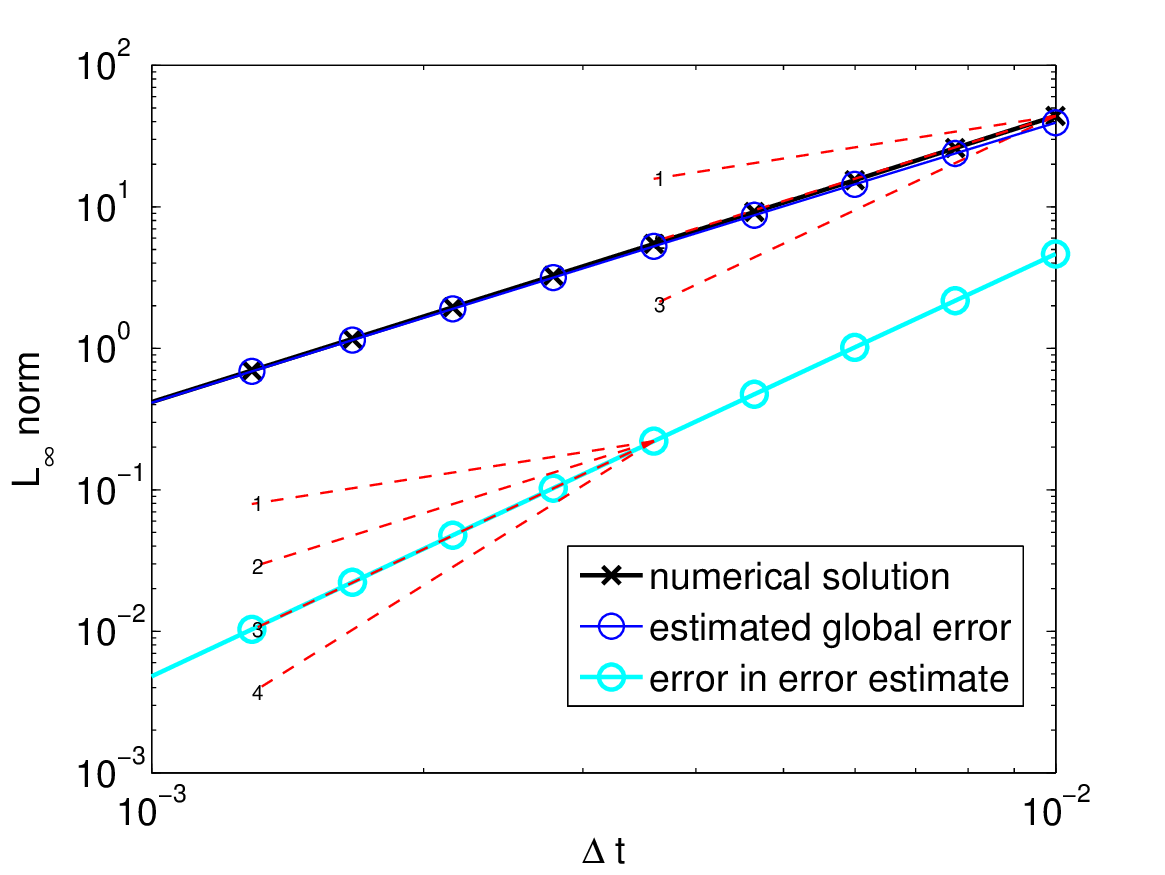}\label{fig:nr:conv:a}}
\subfigure[Third-order method  \eqref{eq:GLyy:third:order}]{\includegraphics[width=0.45\textwidth]{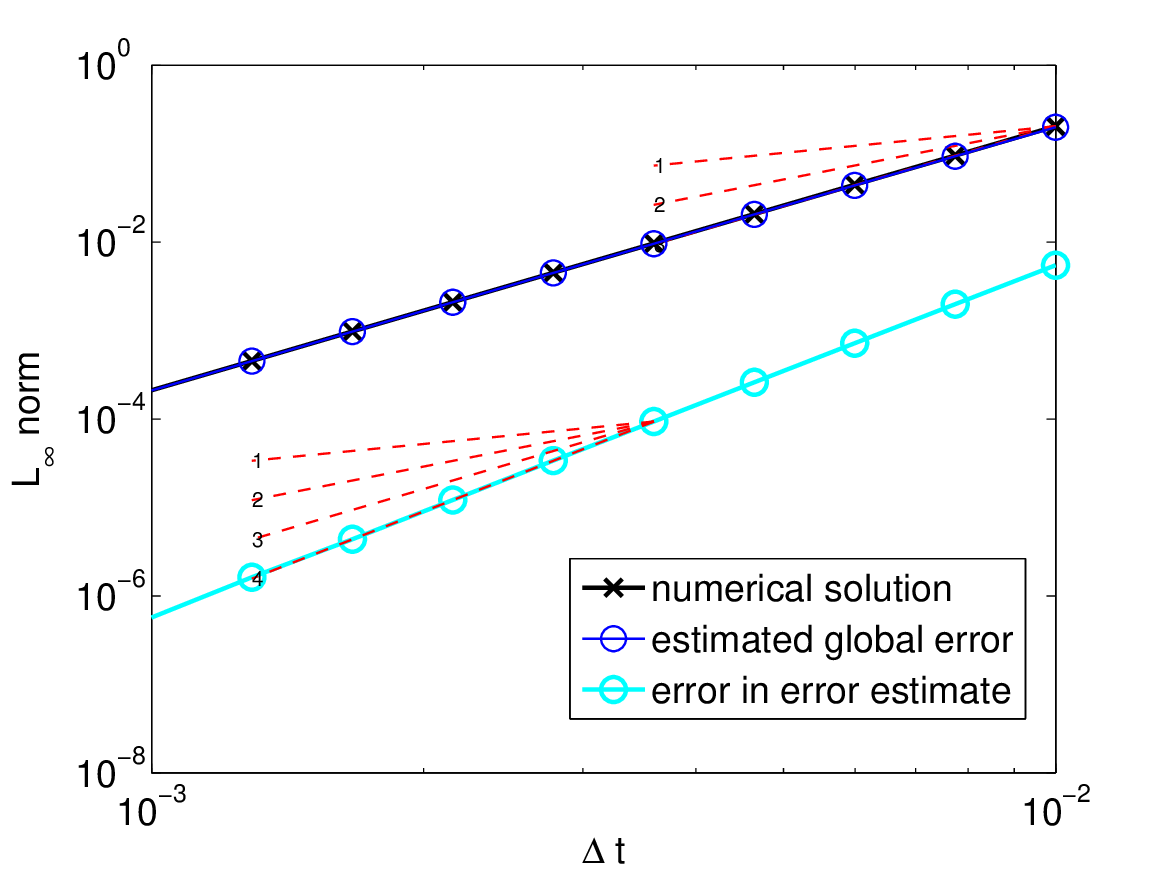}\label{fig:nr:conv:b}}
\end{center}
\caption{Convergence of GEE methods for problem [Prince42]
 \eqref{eq:Prince_1978_(4.2)}. In  {\rm (a)}  we show the error in the
  solution obtained by the second-order GEE method \eqref{GLM:E:s2:p2:RK:A4} and the
  estimated global error, which follows it closely, as well as the
  difference between the true error and the error estimate. An
  asymptotic guide is provided by the red dashed lines.  {\rm (b)}
  This is the same as {\rm (a)} but using the third-order method
  \eqref{eq:GLyy:third:order}. \label{fig:nr:conv}} 
\end{figure}
\begin{figure}
\begin{center}
\subfigure[{Error estimation for methods with exact principal error equation for problem [Prince42]}]{\includegraphics[width=0.45\textwidth]{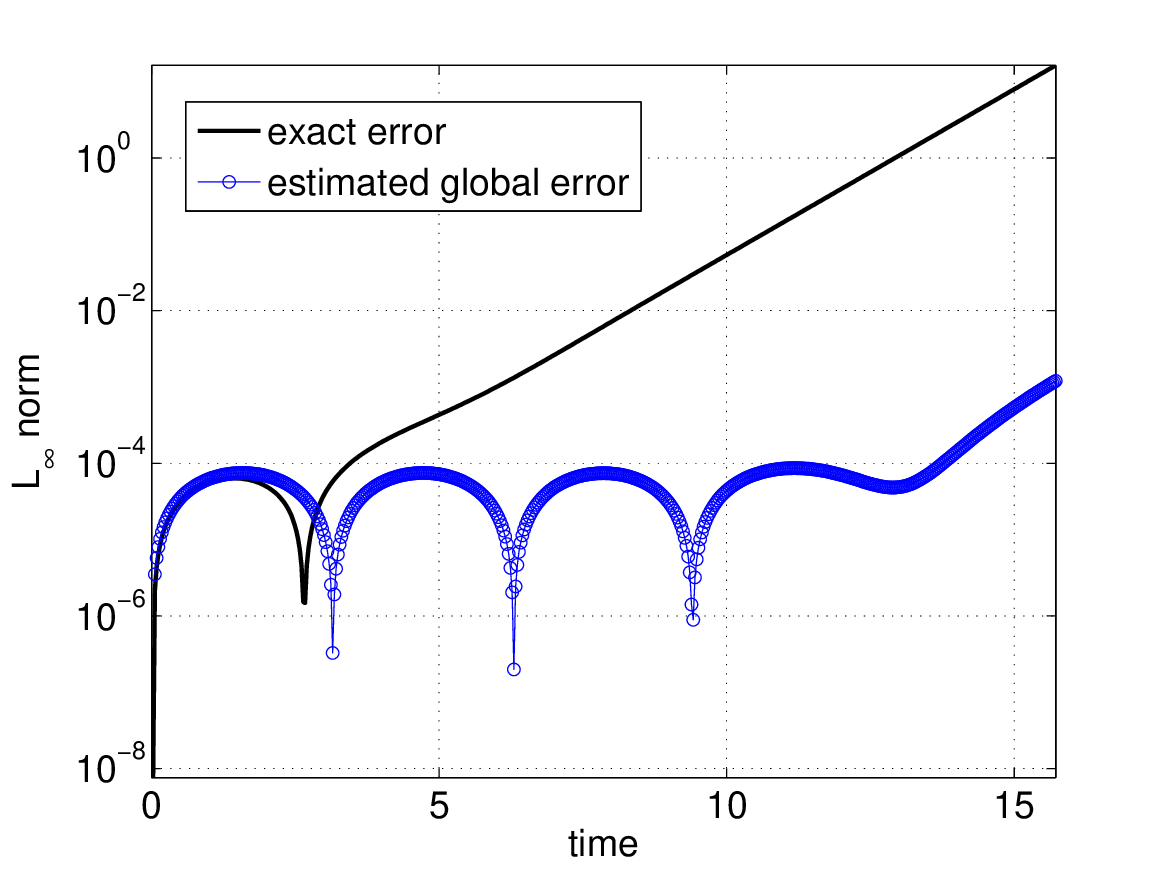}\label{fig:nr:conv:principal:error:a}}
\subfigure[{Method \cite[(3.11)]{Prince_1978}}]{\includegraphics[width=0.45\textwidth]{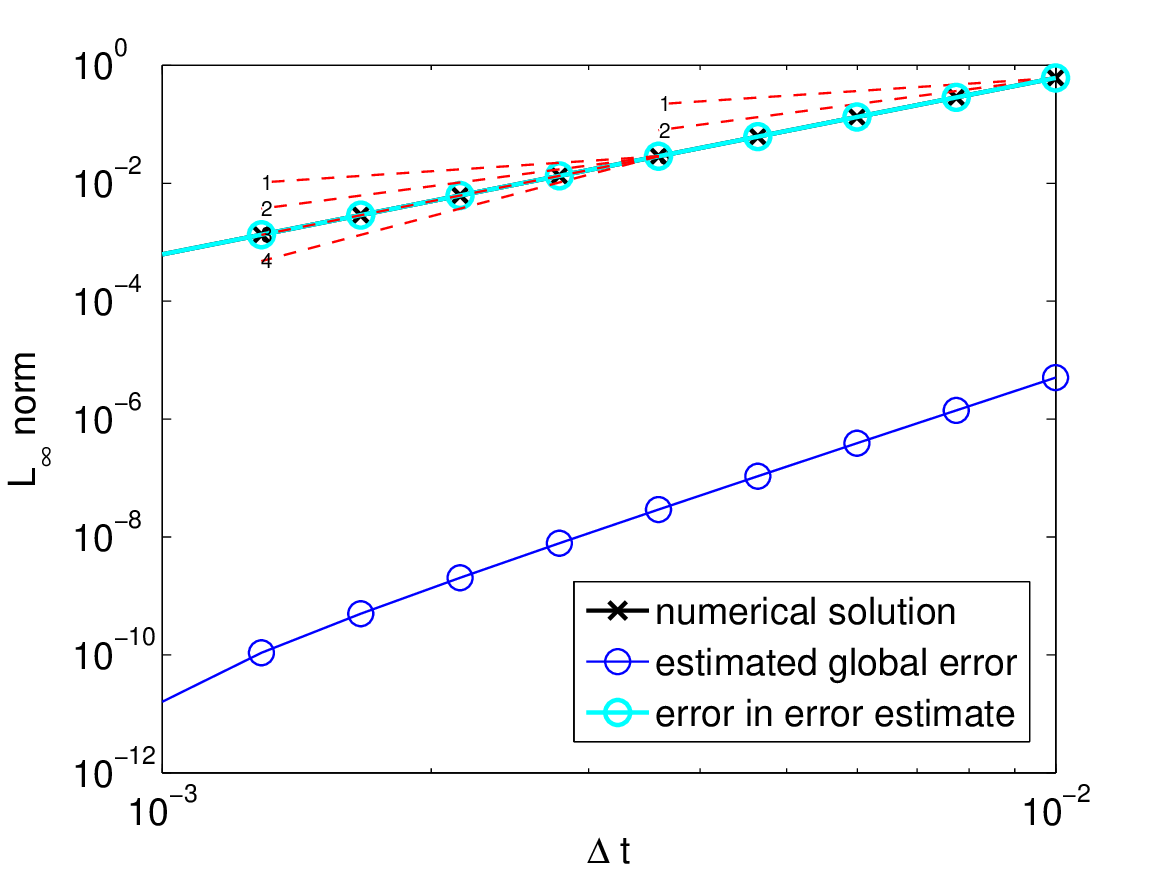}\label{fig:nr:conv:principal:error:b}}
\caption{ {\rm (a)} Failure to capture the
  global errors correctly for problem [Prince42]
 \eqref{eq:Prince_1978_(4.2)} when using a method with exact error equation
 such as \cite[(3.11)]{Prince_1978} with $\dt=3e-02$ and  {\rm (b)}
 its convergence analysis. 
\label{fig:nr:conv:principal:error}}    
\end{center}
\end{figure}

In Fig. \ref{fig:cond:error:error} we illustrate the effect of having
different coupling conditions satisfied. In particular, if $\GLMB\GLMU$
is not diagonal, the GEE may not converge at
all (see Fig. \ref{fig:cond:error:a}). This method is obtained by
replacing the first row of $\GLMB_{y\tilde{y}}$ by [0,0,1] and
$\GLMU_{y\tilde{y}}=\scriptsize{\begin{bmatrix} -4 & 1 & 3/4 \\ 5 & 0 & 1/4 \end{bmatrix}}^T$ in
\eqref{GLM:E:s2:p2:RK:1:GLyy}. If $\GLMB\GLMU$ is diagonal but
$\GLMB\GLMA\GLMU$ is not, the method may converge more slowly to the
asymptotic regime as the high-order error coupling terms become
negligible (see Fig. \ref{fig:cond:error:b}). In
Fig. \ref{fig:cond:error:c}, we show the convergence for the same
problem when both $\GLMB\GLMU$ and $\GLMB\GLMA\GLMU$ are diagonal. 
\begin{figure}
\begin{center}
\subfigure[similar to \eqref{GLM:E:s2:p2:RK:1:GLyy}, $\cancel{\GLMB\GLMU}$, $\cancel{\GLMB\GLMA\GLMU}$]{\includegraphics[width=0.30\textwidth]{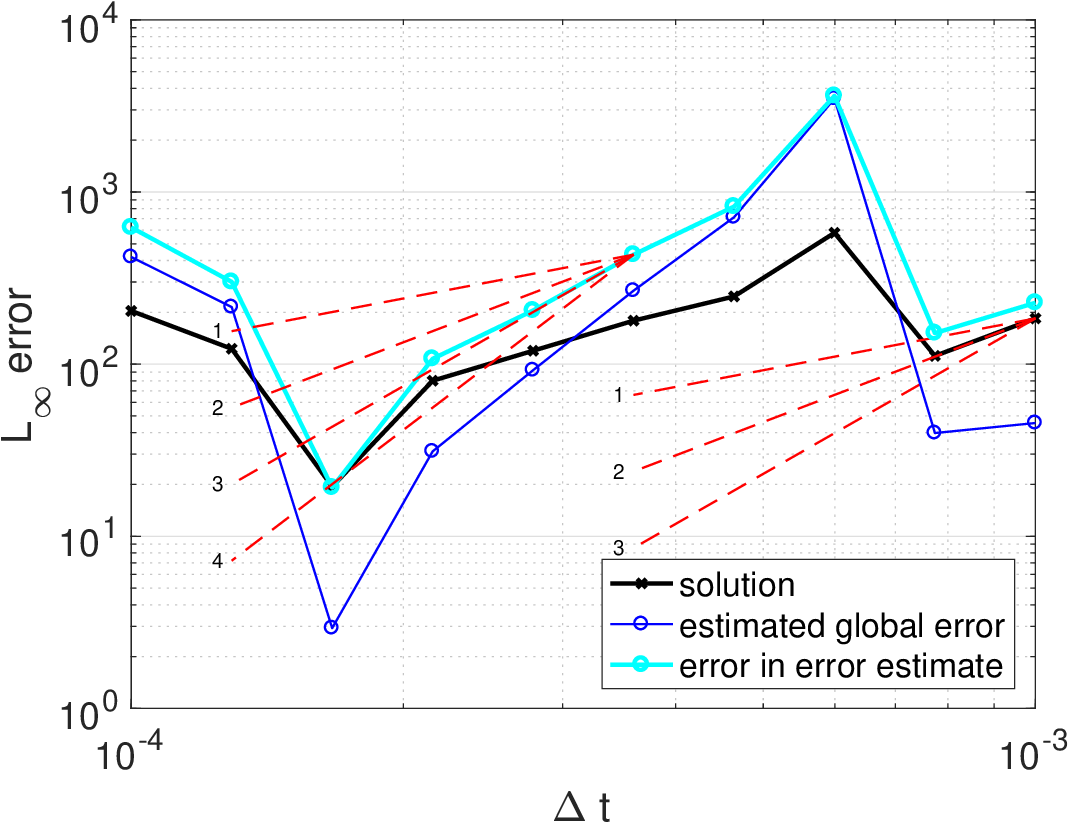}\label{fig:cond:error:a}}
\subfigure[\eqref{GLM:E:s2:p2:RK:A2}, $\GLMB\GLMU$, $\cancel{\GLMB\GLMA\GLMU}$]{\includegraphics[width=0.30\textwidth]{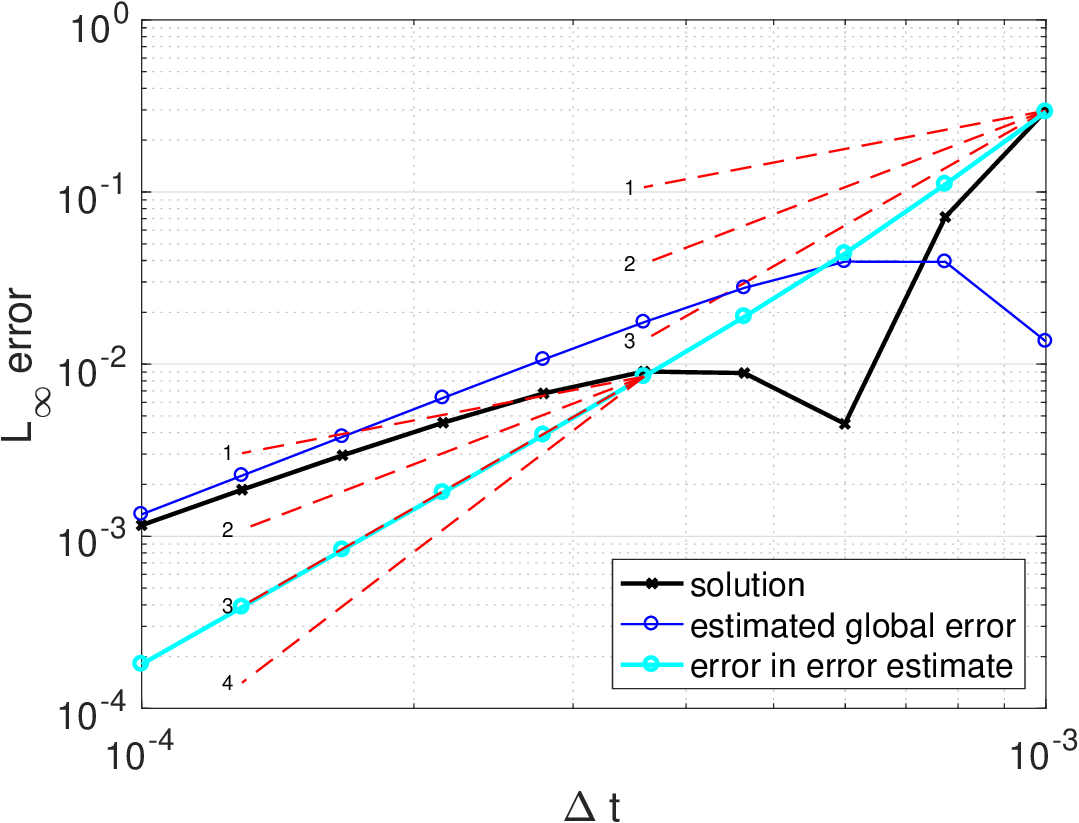}\label{fig:cond:error:b}}
\subfigure[\eqref{eq:GLMyy:A9}, $\GLMB\GLMU$, $\GLMB\GLMA\GLMU$]{\includegraphics[width=0.30\textwidth]{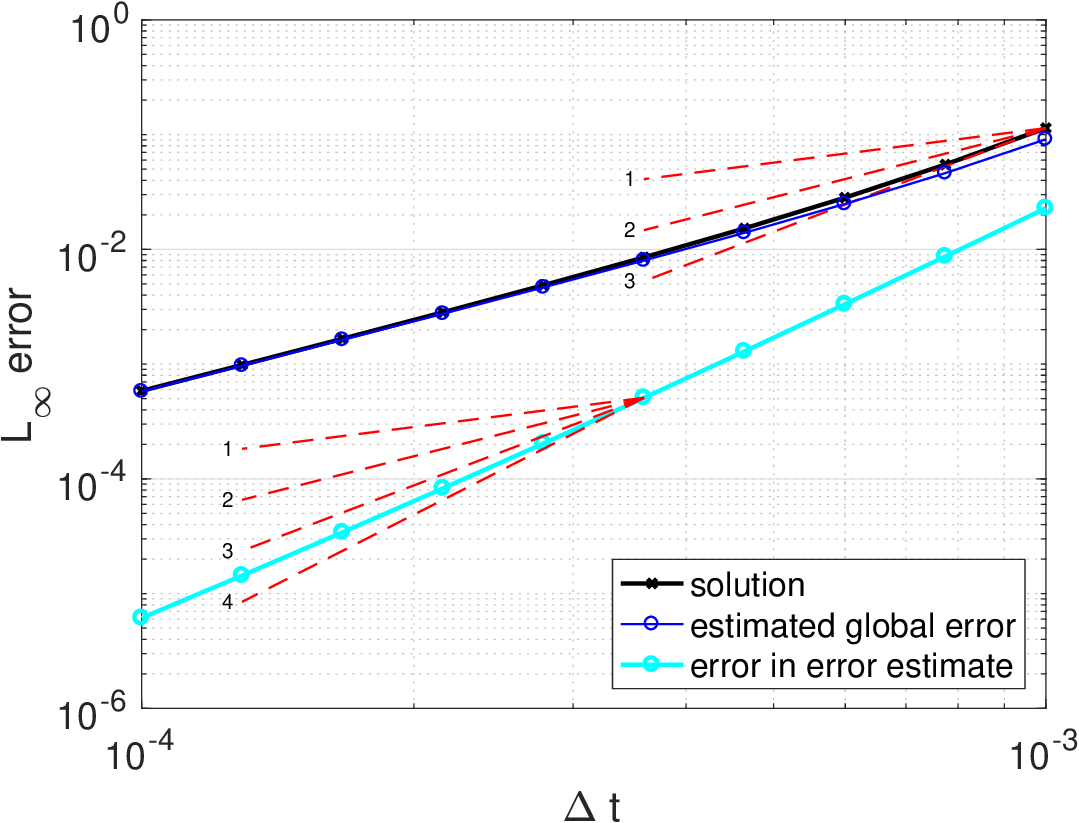}\label{fig:cond:error:c}}
\caption{Convergence of second-order GEE methods with different
  decoupling properties for [Kulikov2013I]. {\rm (a)} Not satisfying $\GLMB\GLMU$ leads to a lack of convergence. {\rm (b)}
  Satisfying just $\GLMB\GLMU$ but not 
  $\GLMB\GLMA\GLMU$ may lead to slower convergence.  {\rm (c)}
  Satisfying both $\GLMB\GLMU$ or
  $\GLMB\GLMA\GLMU$ leads to consistent convergence properties.
\label{fig:cond:error:error}}    
\end{center}
\end{figure}

\paragraph{Linear stability} We next look at the linear stability properties of the methods
introduced in this study. In Fig. \ref{fig:nr:stab:a} we delineate the
stability regions according to \eqref{eq:GL:StabMatrix}. In
Fig. \ref{fig:nr:stab:b} we show numerical results for problem
[LStab2] \eqref{eq:lstab2} with $\lambda \dt
=\{\frac{1}{4},\frac{1}{2},\frac{3}{4},1\}\times(-1 \pm \sqrt{-1})$
when using method \eqref{GLM:E:s2:p2:RK:A2}. As expected, all
solutions except the one corresponding to $\lambda \dt = -1 \pm
\sqrt{-1}$ are 
stable, as can be interpreted from Fig. \ref{fig:nr:stab:a}. 
\begin{figure}
\begin{center}
\subfigure[Stability regions for second-order methods
  \eqref{GLM:E:s2:p2:RK:1:det:GLyy}, \eqref{GLM:E:s2:p2:RK:A2}, \eqref{GLM:E:s2:p2:RK:A4}, 
  and third-order \eqref{eq:GLyy:third:order}]{\includegraphics[width=0.43\textwidth]{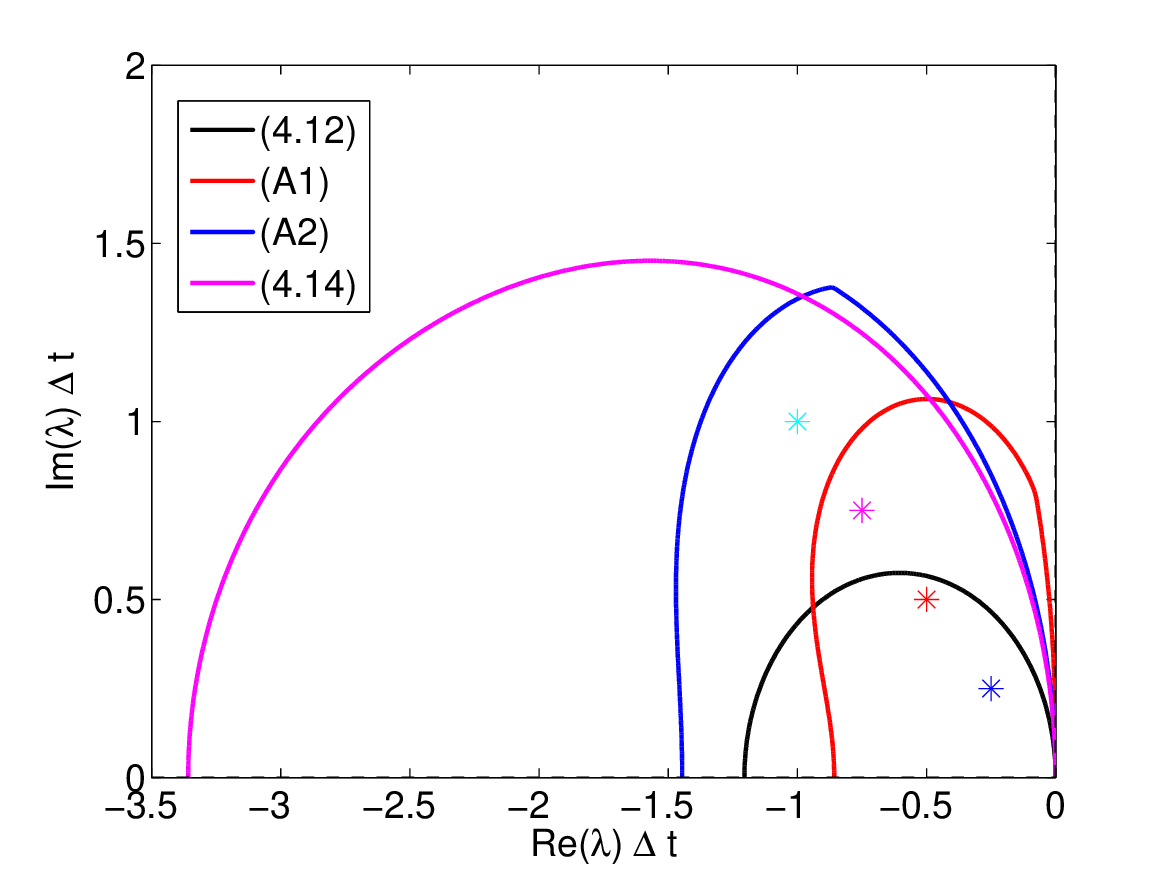}\label{fig:nr:stab:a}}
\subfigure[{Solution of [LStab2] \eqref{eq:lstab2} when using
  method \eqref{GLM:E:s2:p2:RK:A2}}]{\includegraphics[width=0.43\textwidth]{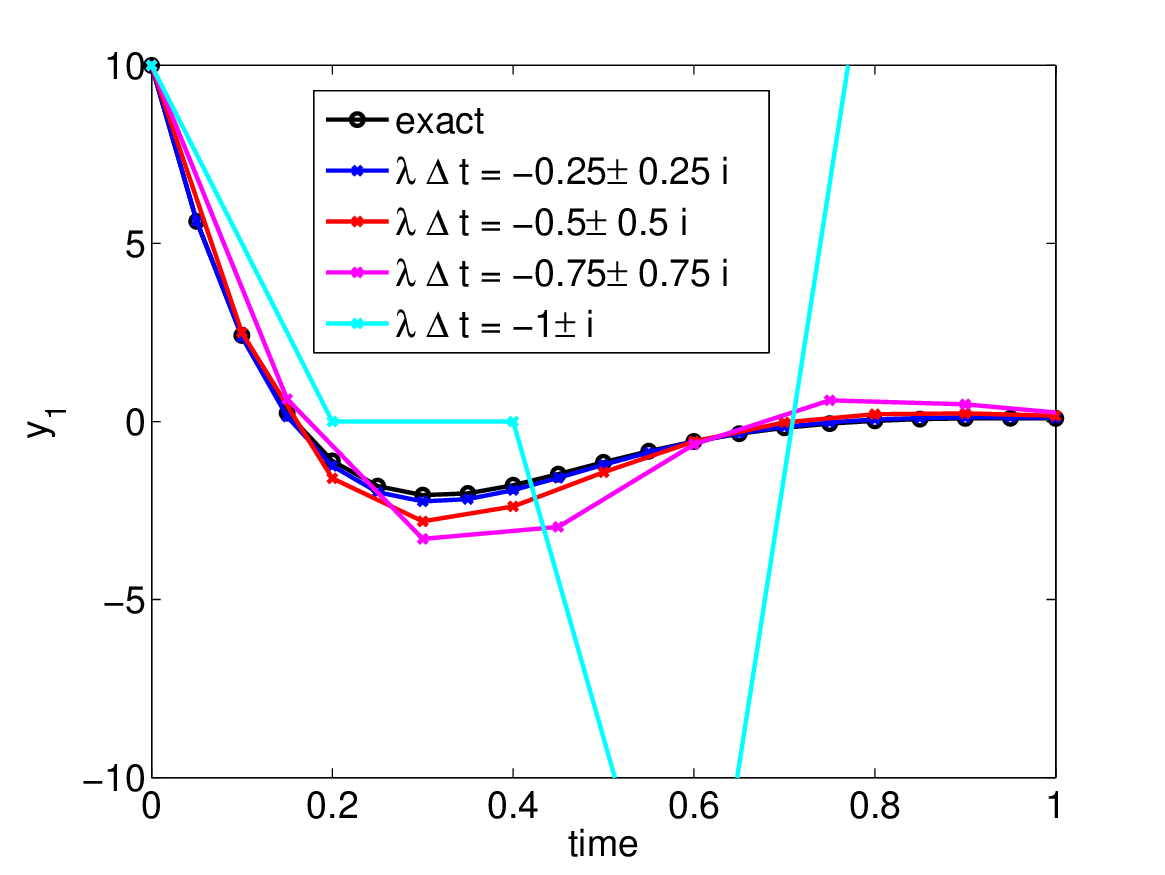}\label{fig:nr:stab:b}}
\end{center}
\caption{Linear stability regions for four  methods introduced
  in this  study  {\rm (a)}; and the solution of method
  \eqref{GLM:E:s2:p2:RK:A2} for problem  [LStab2] \eqref{eq:lstab2}
  in  {\rm (b)} with different time steps. Spectral radius of
  \eqref{GLM:E:s2:p2:RK:A2} is indicated in  {\rm (a)} with marker
  $*$ of the same color. Solutions in {\rm (b)} are stable except the one (cyan line)
  for which the eigenvalues are outside the corresponding stability
  region (red line) in {\rm (a)}.\label{fig:nr:stab}} 
\end{figure}

\paragraph{Adaptive time stepping and practical considerations} Two practical situations are illustrated next: $(i)$ a situation in
which local error control is used to adapt the 
time steps and the global error estimate is used to validate the
solution accuracy and $(ii)$ a situation in which a  prescribed level
of accuracy is needed, which requires recomputing the solution with a
different time step. In the first case we consider a simple local
error control using local error estimates \eqref{eq:GE:GL:local:error}
provided by GEE \eqref{eq:GLyy:third:order}. The local
  error control is achieved by using standard approaches \cite[II.4]{Hairer_B2008_I}. In
Fig. \ref{fig:nr:err:adapt} we show the error evolution [Kulikov2013I]
while attempting to restrict the local error to be around 1e-05. The
time-step length is allowed to vary between 1e-03 and 1e-05. Note that
the global error estimate remains faithful under the time-step change
as expected based on the discussion at the end of Sec. \ref{sec:Global:Errors} 
following
\eqref{eq:global:error:ode:VarT}. We also show that summing the local
errors via cumulative sum (cumsum) yields precisely the global error
as suggested by \eqref{eq:GE:GL:local:error}.   
Similar conditions are used in the second case in which the
target is to satisfy a prescribed tolerance. In a first run the time
step is fixed to 0.0001, and an integration is carried out to the final
time. The global error and the asymptotic are used to determine the time
step that would guarantee a global error of 1e-04. A simple asymptotic analysis
indicates that a time step of 0.00013 needs to be used in order to achieve
that target for GEE \eqref{eq:GLyy:third:order}. The solution is
recomputed by using this smaller fixed time-step, and the global errors
shown in Fig. \ref{fig:nr:err:control} satisfy the prescribed global
accuracy. We note that this case is just an illustration of the
theoretical properties of the methods introduced here. It is likely
a suboptimal strategy for achieving a prescribed global tolerance.

\begin{figure}
\begin{center}
\subfigure{\includegraphics[width=0.46\textwidth]{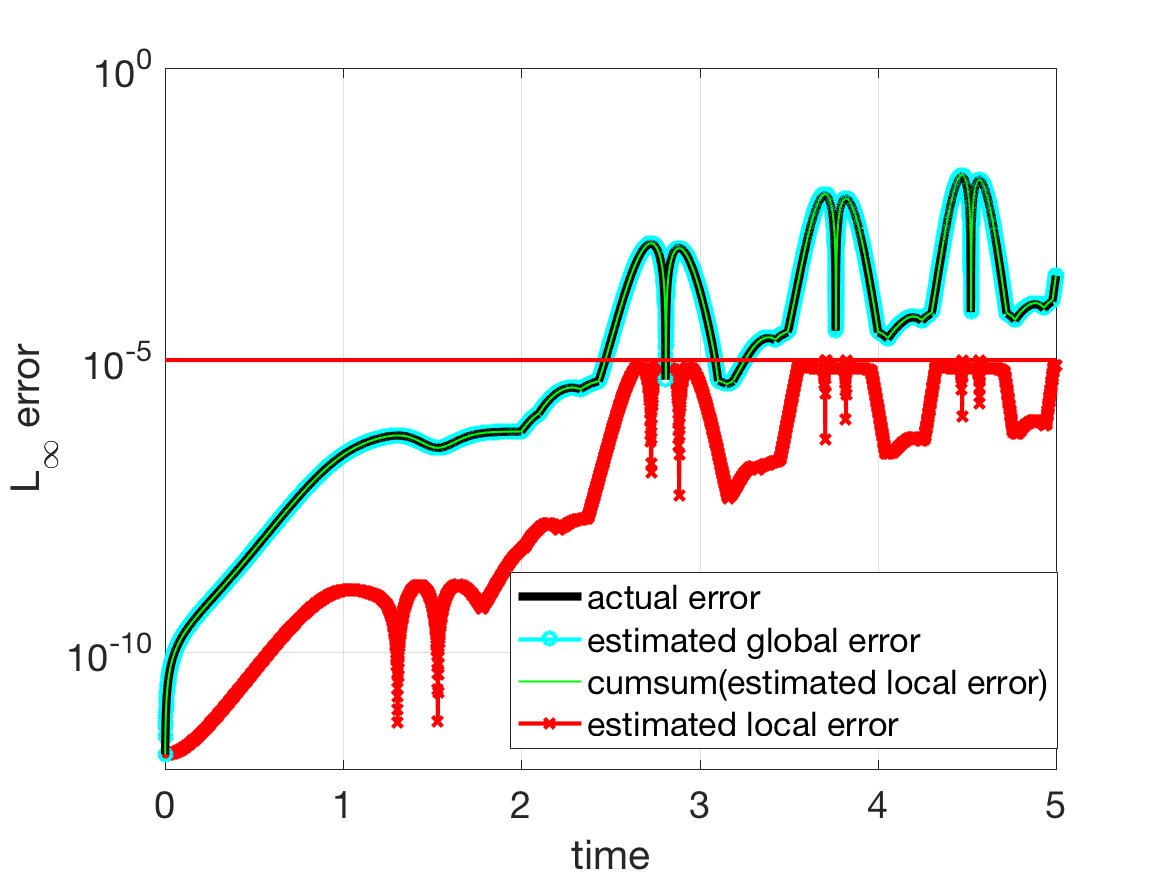}\label{fig:nr:err:adapt:a}}
\subfigure{\includegraphics[width=0.42\textwidth]{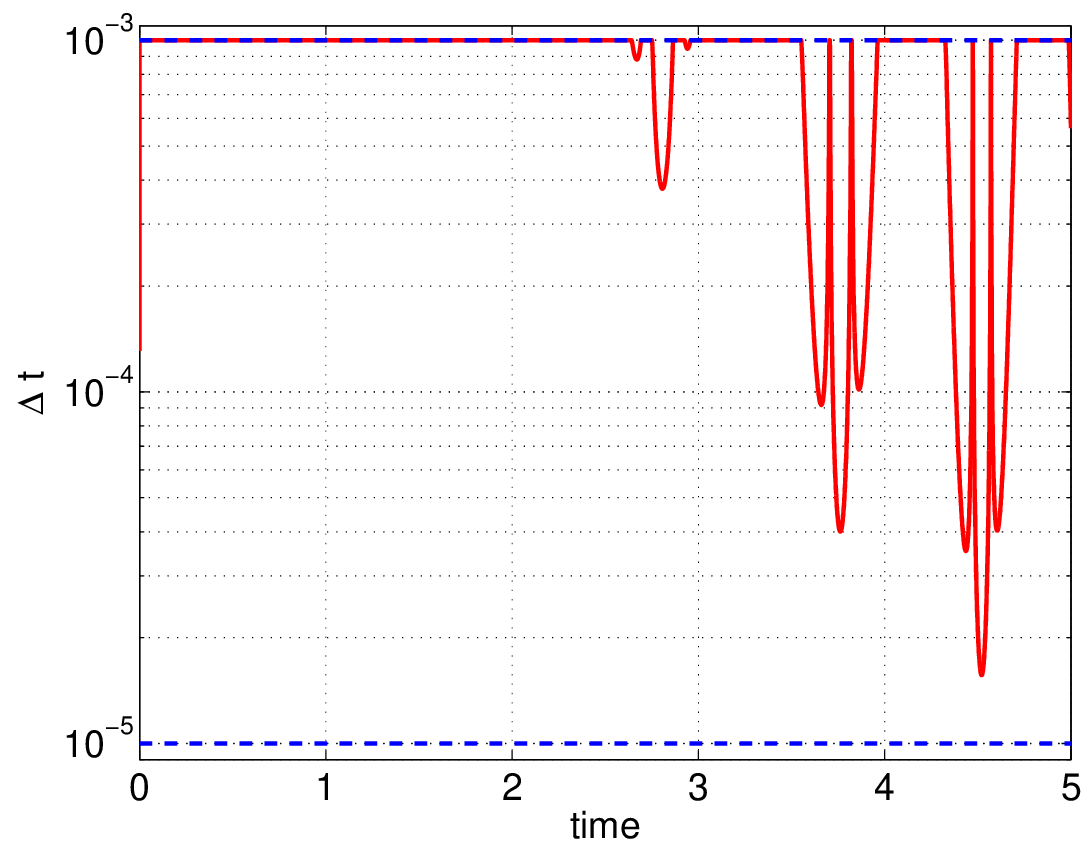}\label{fig:nr:err:adapt:b}}
\end{center}
\caption{Time-step adaptivity using a local error controller. The GEE
  method \eqref{eq:GLyy:third:order} is used to integrate problem
  [Kulikov2013I] by restricting local error to be less than
  1e-05. The error evolution is shown on the left, and the time-step
  length is shown on the right panel. The
time-step length is allowed to vary between 1e-03 and
1e-05.\label{fig:nr:err:adapt}}
\end{figure}
\begin{figure}
\begin{center}
\subfigure[]{\includegraphics[width=0.44\textwidth]{error-all-0.001-Kulikov-2013-I-GLM-A7-globa-0.0001.eps}\label{fig:nr:err:control:a}}
\subfigure[]{\includegraphics[width=0.44\textwidth]{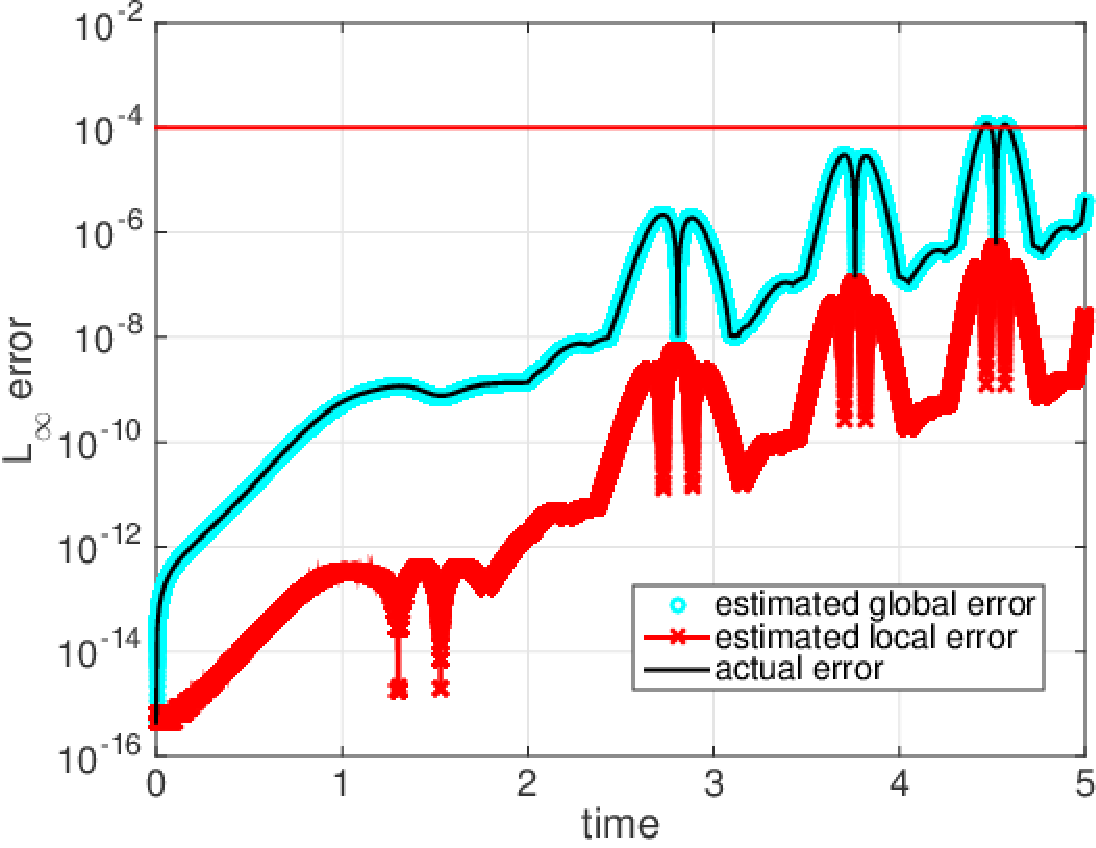}\label{fig:nr:err:control:b}}
\end{center}
\caption{Error evolution for [Kulikov2013I] when using the GEE
  method \eqref{eq:GLyy:third:order} with $\dt=0.001$ (left) and $\dt=
  0.00013$ (right). The smaller time step has been calculated based
  on the GEE global error estimate to keep the global errors below
  1e-04.\label{fig:nr:err:control}}
\end{figure}
%

%
\section{Discussion\label{sec:discussion}}
%

In this study we introduce a new strategy for global error estimation
in time-stepping methods. This strategy is based on advancing in time
the solution along with the defect or, equivalently, two solutions that
have a fixed relation between their truncation errors. The main idea
is summarized in Theorem \ref{prop:two:methods}, and practical
considerations are brought up by Proposition
\ref{prop:GL:independence}. This translates in a particular relation
among the truncation errors of GL outputs and in a decoupling
constraint. We note that this strategy can be seen as 
a generalization of the procedure for solving for the correction and of
several others from the same class. We provide equivalent
representation of these methods in the proposed GL form
\eqref{eq:GLM:Global}.

We have explored several algorithms in this study. The methods
[A:ExPrErEq] with exact principal error equation \eqref{eq:ExactPrincipalError}
\cite{Stetter_1971} are attractive because they offer global
error estimates extremely cheaply; however, they were shown in
\cite{Prince_1978} to be unreliable as illustrated in
Fig. \ref{fig:nr:conv:principal:error}. Strategies that directly
solve the error equation, such as [A:SoErEq] \eqref{eq:GE:LangVerwer}, need a
reliable way of estimating the local errors and the availability of
the Jacobian. We found these methods to be robust, especially
the strategy proposed in \cite{Lang_2007} for low-order methods. The
procedure for solving for the correction [A:SolCor]
\eqref{eq:Solving:Correction:ODE} is arguably one of the most popular
approaches for global error estimation. It is related to [A:ZaPr] and
[A:SoErEq], as discussed, and a 
particular case of this approach is introduced in this study. The
extrapolation algorithm [A:Ex] \eqref{eq:GE:Extrap} is the most robust;
however, it is also the most expensive and also a particular case of
[A:GLMGEE]. With respect to global error control, Kulikov and Weiner
\cite{Kulikov_2013,Weiner_2014} report very promising results. 

The methods introduced here are based on a general linear
representation. The particular case under study is given by form
\eqref{eq:GLM:Global}; however, the analysis is not restricted to that 
situation. Particular instances of second and third order are
presented throughout this study. The error estimates can be used for
error control; however, in this study we do not address this
issue. The GL representation allows stages to be reused or shared
among the 
solutions that are propagated within. This leads to methods having
lower costs. Moreover, the stability analysis is simplified when
compared with global error strategies that use multiple formulas or
equations. We consider only GL methods with two solutions, but this
concept can be extended to having multiple values propagated in time
(e.g., multistep-multistage or peer methods).
 
We provide several numerical experiments in which we illustrate the
behavior of the global error estimation procedures introduced here,
their convergence behavior, and their stability properties.

We investigate nonstiff differential equations. Additional
constraints are necessary in order to preserve error estimation and
avoid order reduction necessary for stiff ODEs.

Global error estimation is typically not used in practice because of
its computational expense. This study targets strategies that would
make it cheaper to estimate the global errors and therefore make them
more practical.  

\appendix

\section{Second-order methods}

\subsection{Other GL second-order methods\label{app:more:second:orders}}
Here we provide two additional second-order methods that we used in
our experiments. A second-order method with $s=3$, $[\GLMB\GLMU,\cancel{\GLMB\GLMA\GLMU}]$, and $\gamma=0$ in
$\GLye$ format is given by
\begin{align}
\label{GLM:E:s2:p2:RK:A2}
\GLMMye=\left[
\begin{array}{ccc|cc}
0     & 0     & 0   & 1  &   4 \\ 
1     & 0     & 0   & 1  &   0 \\
4/9   & 2/9   & 0   & 1  &   0 \\
\hline
0     & -1/2  & 3/2 & 1  &   0\\
1/4   &  1/2  &-3/4 & 0  &   1\\
\end{array}
\right]\,.
\end{align}

Another second-order method with $s=3$ and $[\GLMB\GLMU,\cancel{\GLMB\GLMA\GLMU}]$ that is based on two second-order
approximations ($\gamma=1/2$) in $\GLye$
format is given by
\begin{align}
\label{GLM:E:s2:p2:RK:A4}
\GLMMye=\left[
\begin{array}{ccc|cc}
0     & 0     & 0   & 1  & -11/10 \\ 
1     & 0     & 0   & 1  & 13/30 \\
4/9   & 2/9   & 0   & 1  & 5/3 \\
\hline
5/12  & 5/12  & 1/6 & 1 &  0\\
-1/4  &  -1/4 & 1/2 & 0 &  1\\
\end{array}
\right]\,.
\end{align}

A second order method that results in having both $\GLMB \GLMU$
and $\GLMB\GLMA\GLMU$ diagonal matrices is a four-stage method with the
following coefficients in $\GLyy$ form:
\begin{align}
\label{eq:GLMyy:A9}
\begin{array}{llll}
\GLMA_{2, 1} = 3/4, &  \GLMA_{3, 1} = 1/4,& \GLMA_{3, 2} = 29/60,\\
\GLMA_{4, 1}=-21/44,& \GLMA_{4, 2}= 145/44,& \GLMA_{4, 3} = -20/11,\\ 
\GLMB_{1,1}= 109/275,& \GLMB_{1, 2} = 58/75,& \GLMB_{1,3} = -37/110,& \GLMB_{1,4} =
1/6, \\
\GLMB_{2,1}=3/11,&  & \GLMB_{2,3} = 75/88,& \GLMB_{2,4} = -1/8, \\
&\GLMU_{2, 1} = 75/58,& \\
\GLMU_{1, 2} = 1,& \GLMU_{2, 2}= -17/58,& \GLMU_{3, 2} = 1,&\GLMU_{4, 2} =1,
\end{array}
\end{align}  
where the rest of the coefficients are zero.

\section{ RK3(2)G1 \cite{Dormand_1984} in GL form \eqref{eq:GLM:Global}\label{app:Dormand}}
Method \eqref{eq:Solving:Correction:GLM} corresponding to  RK3(2)G1 \eqref{eq:RK3Q:in}
\cite{Dormand_1984} results in the following  tableau in $\GLye$ form:
\begin{align}
\label{eq:RK3Q:in:GLM}
&
\GLMMye=
\left[
 \begin{array}{cccccccc|cc}
     0&   0&   0&   0&   0&   0&   0& 0 & 1 & 0\\
   1/2&   0&   0&   0&   0&   0&   0& 0 & 1 & 0\\
    -1&   2&   0&   0&   0&   0&   0& 0 & 1 & 0\\
   1/6& 2/3& 1/6&   0&   0&   0&   0& 0 & 1 & 0\\
     0&   0&   0&   0&   0&   0&   0& 0 & 1 & 1\\
 -7/24& 1/3&1/12&-1/8& 1/2&   0&   0& 0 & 1 & 1\\
   7/6&-4/3&-1/3& 1/2&  -1&   2&   0& 0 & 1 & 1\\
     0&    0&  0&   0& 1/6& 2/3& 1/6& 0 & 1 & 1\\
\hline
   1/6& 2/3& 1/6& 0& 0& 0& 0& 0 & 1 & 0\\
  -1/6& -2/3& -1/6& 0& 1/6&2/3& 1/6& 0 &0 & 1\\
\end{array}\right]\,.
\end{align}

\section{Second-order method with
exact principal error equation\label{sec:ExactPrincipalError:Prince_1978}}
The following method is of type \eqref{eq:ExactPrincipalError} and
introduced in  \cite[(3.11)]{Prince_1978}: 
\begin{align}
\label{eq:Prince:1988:2}
\GLMS:=
\begin{array}{c|ccc}
\ST 0          &         0\\
\ST \frac{1}{2}& \frac{1}{2}   &   0 \\
\ST  \frac{5}{8}         &   0          &   \frac{5}{8}        &0\\
\cline{1-4}
\ST &  -\frac{1}{30} &\frac{1}{2} &\frac{8}{15}
\end{array}
\,,~~
\GLMM:=
\begin{array}{c|ccc}
\ST 0          &         0\\
\ST \frac{1}{2}& \frac{1}{2}   &   0 \\
\ST  \frac{3}{4}         &   \frac{1}{2}         &   \frac{1}{4}        &0\\
\cline{1-4}
\ST &  \frac{2}{3} &-1 &\frac{4}{3}
\end{array}
\,,~~
\GLMF:=
\begin{array}{c|ccc}
\ST 0          &         0\\
\ST \frac{1}{2}& \frac{1}{2}   &   0 \\
\ST  \frac{3}{4}         &   \frac{1}{2}         &   \frac{1}{4}        &0\\
\cline{1-4}
\ST &  -\frac{29}{42} &-\frac{31}{42} &\frac{22}{21}
\end{array}
\,.
\end{align}




\end{document}